\documentclass[11pt, reqno]{amsart}
\usepackage{amsmath,amsthm,amssymb,mathrsfs,graphicx, enumitem, hyperref, textcomp}
\usepackage{soul}

\usepackage[all,cmtip,pdf]{xy}

\usepackage{xcolor}
\usepackage{tikz}
\usetikzlibrary{positioning}
\usepackage{etoolbox}
\let\bbordermatrix\bordermatrix
\patchcmd{\bbordermatrix}{8.75}{4.75}{}{}

\newcommand{\U}{\mathcal{U}}

\newcommand{\E}{\mathcal{E}}
\newcommand{\F}{\mathcal{F}}
\newcommand{\Ell}{\textup{Ell}}

 \newcommand{\id}{\mathrm{id}}
\newcommand{\Aut}{\mathrm{Aut}}
\newcommand{\supp}{\mathrm{supp}}

\newcommand{\dist}{\mathrm{dist}}
\newcommand{\ocap}{\mathrm{ocap}}
\newcommand{\prim}{\mathrm{Prim}}
\newcommand{\diag}{\mathrm{diag}}

\newcommand{\mdim}{\mathrm{mdim}}

\theoremstyle{plain}
\newtheorem{theorem}{Theorem}[section]
\newtheorem{lemma}[theorem]{Lemma}
\newtheorem{corollary}[theorem]{Corollary}
\newtheorem{proposition}[theorem]{Proposition}

\newtheorem*{theorem*}{Theorem}
 \newtheorem*{lemma*}{Lemma}
\theoremstyle{definition}
\newtheorem{definition}[theorem]{Definition}
\newtheorem{example}[theorem]{Example}
\newtheorem{notation}[theorem]{Notation}

\newtheorem{remark}[theorem]{Remark}
\newtheorem{remarks}[theorem]{Remarks}
\theoremstyle{remark}

\author[Forough, Jeong, Strung]{Marzieh Forough \and Ja A Jeong \and Karen R. Strung}

\address{Department of Applied Mathematics, Faculty of Information Technology\\ Czech Technical University in Prague\\ Th\'akurova 9 \\160 00, Prague 6, Czech Republic, and}
\address{Department of Abstract Analysis\\ Institute of Mathematics, Czech Academy of Sciences, \v{Z}itn\'a 25, 115 67 Prague 1, Czech Republic}
\email{foroumar@fit.cvut.cz}

\address{Department of Mathematical Sciences and Research Institute of Mathematics\\
Seoul National University\\
Seoul 08826, Korea}
\email{jajeong@snu.ac.kr}

\address{Department of Abstract Analysis\\ Institute of Mathematics, Czech Academy of Sciences, \v{Z}itn\'a 25, 115 67 Prague 1, Czech Republic}
\email{strung@math.cas.cz}

\title[$\mathrm{C}^*$-algebras of homeomorphisms twisted by line bundles]{$\mathcal{Z}$-stability for $\mathrm C^*$-algebras of minimal line-bundle-twisted homeomorphisms with the small boundary property}

\thanks{KRS and MF are funded by GAČR project GF25-15403K. KRS is also funded by RVO: 67985840. Part of this work was carried out when KRS was funded by GAČR project 20-17488Y and when JAJ was partially supported by NRF 2018R1D1A1B07041172.}

\begin{document}

\begin{abstract} In this paper we show that the Cuntz--Pimsner algebras associated to minimal homeomorphisms twisted by line bundles, along with their orbit-breaking subalgebras, are $\mathcal{Z}$-stable whenever the underlying dynamical system has the small boundary property. This entails that this class is classified by the Elliott invariant. Furthermore, we show that the tensor product of two such $\mathrm{C}^*$-algebras is always $\mathcal{Z}$-stable, without assuming the small boundary property. In particular this applies to  $\mathrm{C}^*$-algebras arising from systems with positive mean dimension.
\end{abstract}

\maketitle

\tableofcontents

\section{Introduction}

A major advancement in the  theory of $\mathrm C^*$-algebras has been the establishment of the classification theorem for simple, separable, unital, nuclear $\mathrm C^*$-algebras that absorb the Jiang--Su algebra $\mathcal Z$ and satisfy the Universal Coefficient Theorem (UCT). This theorem states that such $\mathrm C^*$-algebras are completely determined by their Elliott invariant, an invariant consisting of $K$-theory, tracial data, and the natural pairing between them. The Jiang--Su algebra is a simple, separable, unital, nuclear $\mathrm C^*$-algebra with the same Elliott invariant as the complex numbers. Since the invariant behaves well with respect to tensor products,  at the level of the invariant  tensoring by $\mathcal Z$ looks just like tensoring by $\mathbb C$. Hence for classification by this invariant to hold  one must restrict attention to algebras $A$ that are $\mathcal Z$-stable, that is, those satisfying  $A \cong A \otimes \mathcal Z$. The UCT, which facilitates the transfer of information between K-theory and KK-theory,  holds for all known nuclear $\mathrm C^*$-algebras, though determining whether nuclearity itself implies the UCT remains a difficult open question.

With the classification theorem now in place, a central task is to determine which naturally-occurring $\mathrm C^*$-algebras fall within its scope. In particular, given a class of $\mathrm C^*$-algebras arising from dynamics, geometry, or algebraic constructions, one would like to identify when its members are simple, separable, unital, nuclear, and $\mathcal Z$-stable, and hence classifiable by the Elliott invariant.

In this paper we study a class of $\mathrm C^*$-algebras that arise as Cuntz–Pimsner algebras associated to Hilbert bimodules over commutative $\mathrm C^*$-algebras. These can be viewed as ``twisted'' crossed products of homeomorphisms, where we replace the homeomorphism with a full Hilbert bimodule and the crossed product by a Cuntz--Pimsner algebra.

More precisely, given a Hermitian line bundle $\mathscr V$ over an infinite compact metric space $X$ and a homeomorphism $\alpha : X \to X$, one constructs a Hilbert $C(X)$-bimodule $\Gamma (\mathscr V, \alpha)$ consisting of continuous sections of $\mathscr V$ with the left and right actions intertwined by $\alpha$. When $\mathscr V$ is a trivial line bundle, the resulting Cuntz–Pimsner algebra $\mathcal O_{C(X))}(\Gamma(X \times \mathbb C, \alpha))$ is exactly the crossed product $C(X) \rtimes_\alpha \mathbb Z$, which has been extensively studied in the context of the classification program. In particular, Elliott and Niu showed that such crossed products are $\mathcal Z$-stable if and only if the underlying minimal dynamical system 
$(X, \alpha)$ has \emph{the small boundary property}, or equivalently,  \emph{mean dimension zero}.

In this paper,  the ``if'' direction of the $\mathcal Z$-stability result is extended to the twisted setting, that is, where the Hilbert bimodule arises from a nontrivial line bundle. This framework naturally generalizes the usual crossed product construction, while introducing subtle geometric features related to the topology of $\mathscr V$. We prove the following main theorem: 
\begin{theorem*}
Let $X$ be an infinite compact metric space, $\alpha : X \to X$ a minimal homeomorphism and $\mathscr V$ a line bundle over $X$. Suppose $(X, \alpha)$ has the small boundary property. Then $\mathcal O_{C(X)}(\Gamma(\mathscr{V}, \alpha))$ is $\mathcal Z$-stable.
\end{theorem*}

The proof uses the notion of so-called \emph{orbit-breaking} subalgebras. By \cite[Theorem 6.16 ]{AAFGJSV2024}, if the orbit of $\alpha$ is ``broken'' at a single point $y \in X$, the corresponding $\mathrm C^*$-subalgebra $\mathcal O_{C(X)}(C_0(X \setminus \{y\})\Gamma(\mathscr{V}, \alpha))$ is \emph{centrally large} (see \cite[Definition 2.1]{ArchPhil:SR1}) in $\mathcal O_{C(X)}(\Gamma(\mathscr{V}, \alpha))$, and so is enough to show $\mathcal Z$-stability of $\mathcal O_{C(X)}(C_0(X \setminus \{y\})\Gamma(\mathscr{V}, \alpha))$. Establishing $\mathcal Z$-stability then relies on approximation by the recursive subhomogeneous (RSH) algebras constructed in \cite{FoJeSt:rsh}. By showing that the RSH algebras involved can always be approximated by subhomogeneous algebras with arbitrarily small dimension ratio (see Definition~\ref{def:dimrat}), we show they are $\mathcal Z$-stable, Theorem~\ref{thm:Z-stable}.  The small boundary property of the underlying dynamical system plays a crucial role in constructing partitions of unity which allow for approximating the RSH algebras in a controlled way. 

In addition, we show that tensor products of such $\mathrm C^*$-algebras are always $\mathcal Z$-stable, even without the assumption of the small boundary property. The proof again relies on decompositions and approximations of orbit-breaking subalgebras.

The structure of the paper is as follows.  Section~\ref{sec:prelim} contains preliminaries on the Cuntz–Pimsner construction for Hilbert $\mathrm{C}^*$-bimodules. We then briefly review definitions and properties of the small boundary property and mean dimension for homeomorphisms. Section~\ref{sec:MainCharacter} introduces the $\mathrm C^*$-algebras associated to minimal homeomorphisms twisted by line bundles, and recalls their orbit-breaking subalgebras. Using results from \cite{AAFGJSV2024} and \cite{FoJeSt:rsh}, Section~\ref{subsection:RSH} establishes structural results and a recursive subhomogeneous decomposition of these algebras, and in Section~\ref{sec:irrep}  we establish results about their irreducible representations. Section~\ref{sec:Zstab} applies these results to prove $\mathcal Z$-stability when the underlying system has the small boundary property. Section~\ref{sec:TP} shows that tensor products of such algebras are always $\mathcal Z$-stable.

\subsubsection*{Acknowledgments} The authors would like to thank Gi Hyun Park and Maria Grazia Viola for helpful discussions. This project was initiated as part of the workshop 18w5168  \emph{Women in Operator Algebras} at the Banff International Research Station in November 2018.


\section{\texorpdfstring{$\mathrm{C}^*$}{C*}-algebraic and dynamical preliminaries} \label{sec:prelim}

\subsection{Hilbert \texorpdfstring{$C^*$}{}-bimodules and Cuntz--Pimsner algebras}
Let $A$ be a  $\mathrm C^*$-algebra. A \emph{Hilbert $A$-bimodule} is a $\mathbb C$-linear space $\E$ which is both a left and right Hilbert $A$-module and satisfies the condition
\[
\xi_1 \langle \xi_2, \xi_3 \rangle_\E = \,_\E\langle \xi_1 , \xi_2 \rangle \, \xi_3, 
\]
for every $\xi_1, \xi_2, \xi_3 \in \E$.

 Let $\E$ be a Hilbert $A$-bimodule. The conjugate module of $\E$ is the Hilbert $A$-bimodule 
\[ \E^* = \{ \widehat{\xi} \mid \xi \in \E\}, \]
with
\[ a \widehat{\xi} := \widehat{\xi a^*}, \quad \widehat{\xi} a := \widehat{a^* \xi}, \quad \widehat{\xi} \in \E^*, a \in A, \]
and 
\[ _{\E^*} \langle\, \widehat{\xi}, \widehat{\eta} \,\rangle :=  \, _\E \langle \eta, \xi \rangle, \quad  \langle\, \widehat{\xi}, \widehat{\eta}\, \rangle_{\E^*}  :=  \,  \langle \eta, \xi \rangle_\E, \quad \xi, \eta \in \E.\]

The $A$-balanced tensor product $\E_1  \otimes_A \E_2$ of two Hilbert $A$-bimodules is again a Hilbert $A$-bimodule in the obvious way. To ease notation, we write 
\[
\E^{\otimes n} := \underbrace{\E \otimes_A \cdots \otimes_A \E}_{n \text{ times}}.
\]
Let $A$ be a $\mathrm C^*$-algebra and $\E$ a Hilbert $A$-bimodule. A \emph{covariant representation} (cf. \cite[Definition 2.1]{Katsura2004}, see also \cite{AEE:Cross}) of $\E$ on a $\mathrm{C}^*$-algebra $B$ is a pair $(\pi, \tau)$ consisting of a $^*$-homomorphism $\pi : A \to B$ and a linear map $\tau : \E \to B$ satisfying  
\begin{enumerate}
    \item $\tau(\xi)^*\tau(\eta) = \pi(\langle \xi, \eta \rangle_A)$,
    \item $\pi(a) \tau(\xi) = \tau(a \xi)$,
\end{enumerate}
for every $\xi, \eta \in \E$, $a \in A$. Note that (2) implies that $\tau(\xi)\pi(a) = \tau(\xi a)$ for every $\xi \in \E$ and $a \in A$. 

Given a covariant representation $(\pi, \tau)$ of $\E$ in a $\mathrm{C}^*$-algebra $B$, let $\mathrm C^*(\pi, \tau)$ denote the $\mathrm C^*$-subalgebra of $B$ generated by $\pi(A)$ and $\tau(\E)$. In \cite[Section 4]{Katsura2004}, Katsura constructs the \emph{universal covariant representation} $(\pi_u, \tau_u)$ of $\E$: for any covariant representation $(\pi, \tau)$ of $\E$, there exists a surjective $^*$-homomorphism $\rho: C^*(\pi_u, \tau_u) \to \mathrm C^*(\pi, \tau)$ such that $\pi = \rho \circ \pi_u$ and $\tau = \rho \circ \tau_u$. 

\begin{definition}[{cf. \cite[Definition 3.5]{Katsura2004}}]
    Let $A$ be a $\mathrm C^*$-algebra and $\E$ a Hilbert $A$-bimodule. The \emph{Cuntz--Pimsner algebra} of $\E$, denote $\mathcal O_A(\E)$ (or simply $\mathcal O (\E)$ if $A$ is understood), is the $\mathrm C^*$-algebra generated by the universal covariant representation of $\E$. \end{definition}

    The Cuntz--Pimsner algebra construction is in fact more general, as one only requires that $\E$ is a $\mathrm{C}^*$-correspondence over $A$. 

    Let $A$ be a $\mathrm{C}^*$-algebra and $\E$ a Hilbert $A$-bimodule. The Cuntz--Pimsner algebra $\mathcal{O}_A(\E)$ admits a gauge action of $\mathbb{T}$ which makes $\mathcal{O}_A(\E)$ into a $\mathbb Z$-graded $\mathrm C^*$-algebra in the sense of \cite[Definition 16.2]{Exel:Book}. We have
    \[
\mathcal{O}_A(\E) = \overline{\bigoplus_{n \in \mathbb{Z}} E_n},
    \]
    where 
    \begin{equation} \label{eq:En} E_n \cong \left\{ \begin{array}{cc} \E^{\otimes n}, & n > 0, \\
    A, & n = 0, \\
    (\E^*)^{\otimes n} & n < 0.\end{array} \right. \end{equation}

Note that for a Hilbert $A$-bimodule $\E$, $\mathcal O(\E)$ is the same as the crossed product $A \rtimes_\E \mathbb Z$ defined by \cite{AEE:Cross}. While the $\mathrm C^*$-algebras in this paper behave a lot like crossed products of $C(X)$ by the integers, we prefer the Cuntz--Pimsner notation as it is more concise.

\subsection{The small boundary property and mean dimension zero}

Let $(X, \alpha)$ be a dynamical system where $X$ is an infinite compact metric space, and let $C$ be a subset of $X$. The \emph{orbit capacity} of the set $C$ is
 \[
 \ocap(C)=\lim_{n \to \infty} \frac{1}{n} \sup_{x \in X} \sum_{i=0}^{n-1}\chi_{C}(\alpha^i(x)).\]

 \begin{definition} \cite[Definition 5.2]{LindWeiss:MTD}
    A dynamical system $(X, \alpha)$ has the \emph{small boundary property} (SBP) if for every point $x \in X$ and every open set $U$ containing $x$, there is a neighborhood $V \subseteq U$ of $x$ with $\ocap(\partial V)=0$.
\end{definition}

An equivalent definition can be given in terms of invariant measures: $(X, \alpha)$ has the \emph{small boundary property} (SBP) if for every point $x \in X$ and every open set $U$ containing $x$, there is a neighborhood $V \subseteq U$ of $x$ such that $\mu(\partial V)=0$ for every $\alpha$-invariant measure $\mu$ \cite[Proposition 1]{ShubWeiss1991}. This is the definition used in  \cite{EllNiu:SBP}, \cite{EllNiu:SBP25}, \cite{KerrSz:almostfinit}), which link the small boundary property for a dynamical system to regularity properties of the associated crossed product $\mathrm C^*$-algebra.

The \emph{mean (topological) dimension} of a topological dynamical system, which can be thought as a dynamical analogue of covering dimension, was first introduced by Gromov, and subsequently studied more intensely by Lindenstrauss and Weiss. This invariant provides information of a dynamical system even for the cases that topological entropy or topological dimension are infinite. From the point of view of $\mathrm{C}^*$-algebras, the mean dimension can further provide information about the structure of $\mathrm C^*$-algebras arising from topological dynamical systems. 

Let $X$ be a compact metric space.  Recall that the \emph{order} of an open cover $\mathcal{U}$ of $X$, is given by
\[
\mathrm{ord}(\U)= \left(\sup_{x \in X} \sum_{U \in \U} \chi_{U}(x)\right)-1.
\]
We say that a cover $\mathcal{W}$ refines the cover $\mathcal{U}$, denoted  $\mathcal{W} \succ \mathcal{U}$, if every member of $\mathcal{W}$ is a subset of a member of $\mathcal{U}$. Let 
\[
\mathcal{D}(\U) :=\inf_{\mathcal{W} \succ \U} \mathrm{ord}(\mathcal{W}).
\]
By \emph{dimension} of a space $X$, denoted $\dim X$, we mean the \emph{topological (covering) dimension} of $X$: given $m \in \mathbb Z_{\geq 0}$, we have $\dim X \leq m$ if and only if $\mathcal{D}(\U) \leq m$ for all covers $\U$ of $X$. We write $\dim X = m$ if $m = \min \{ n \in \mathbb Z_{\geq 0} \mid \dim X \leq n\}$. If no such $m$ exists, we put $\dim X := \infty$.

For two open covers $\U$ and $\mathcal{W}$ of $X$, let $\U \vee \mathcal{W}= \{ U \cap W \mid U \in \U,  W \in \mathcal{W}\}$. Let $(X, \alpha)$ be a topological dynamical system. For every positive integer $n$ and cover $\U$ of $X$, we denote by $\U^n$ the cover
\[ \U^n:=\U \vee \alpha^{-1}\U \vee \cdots \vee \alpha^{-n}\U.
\]

\begin{definition}[{\cite[Section 0.4]{GromovMD}, \cite[Definition 2.6]{LindWeiss:MTD}}]
    Let $(X, \alpha)$ be a topological dynamical system. The \emph{mean dimension of $(X, \alpha)$} denoted by $\mdim(X, \alpha)$ (or $\mdim(X)$ if $\alpha$ is understood) is given by 
    \[
\mdim(X, \alpha)= \sup_{\U} \lim_{n \to \infty} \frac{\mathcal{D}(\U^{n-1})}{n},
    \]
    where $\U$ runs over all finite open covers of $X$.
\end{definition}
Observe that for any two finite open covers $\U$ and $\mathcal{W}$ of $X$, we have $\mathcal{D}(\U \vee \mathcal{W}) \leq \mathcal{D}(\U)+\mathcal{D}(\mathcal{W})$, so the above limit exists and mean dimension is well defined.

Any dynamical system $(X, \alpha)$ with the small boundary property has mean dimension zero~\cite[Theorem 5.4]{LindWeiss:MTD}. Conversely, if $(X, \alpha)$ has mean dimension zero and is the extension of a minimal system, than $(X, \alpha)$ has the small boundary property~\cite[Theorem 6.2]{Lin:MD-SEF}. In particular, mean dimension zero and the small boundary are equivalent for minimal dynamical systems.

The key role of the small boundary property in Theorem~\ref{thm:FindingS} is in construction of partitions of unity with the following property:

\begin{proposition}\cite[Proposition 5.3]{LindWeiss:MTD} \label{prop:SBP-POU}
 If $(X, \alpha)$ has the SBP then for every finite open cover $\U$ of $X$ and every $\epsilon >0$ there is a partition of unity $(\phi_U)_{U\in \U}$ subordinate to cover $\U$ such that 
 \[
 \ocap(\bigcup_{U \in \U} \phi_U^{-1}(0,1)) < \epsilon.
 \]
\end{proposition}

The class of minimal dynamical systems with the small boundary property, or equivalently, mean dimension zero is quite large. Some examples include:
\begin{enumerate}
    \item When the dimension of $X$ is finite then $\mathcal{D}(\U^n) \leq \dim(X)$ for all finite open cover $\U$ and positive integers $n$. Hence $\mdim(X, \alpha)=0$ which implies the small boundary property.
     \item If the topological entropy of $(X, \alpha)$ is finite, then $\mdim(X, \alpha) = 0$ \cite[Section 4]{LindWeiss:MTD}. 
    \item If $(X, \alpha)$ has at most countably many ergodic probability measures, then $(X, \alpha)$ has the small boundary property \cite{ShubWeiss1991}. In particular, any uniquely ergodic system has the small boundary property.
    \item Let $\mathcal{M}^1(X, \alpha)$ denote the convex set of all $\alpha$-invariant Borel probability measures on $X$ equipped with the weak*-topology. Then $\mathcal{M}^1(X, \alpha)$ is a Choquet simplex and the extreme points correspond to the set of ergodic measures. If $\mathcal{M}^1(X, \alpha)$ is a Bauer simplex (that is, the set of its boundary points is closed), then $(X, \alpha)$ has the small boundary property \cite[Theorem 4.6]{EllNiu:SBP}.
\end{enumerate}

\section{\texorpdfstring{$\mathrm{C}^*$}{C*}-algebras associated to minimal homeomorphisms twisted by line bundles} \label{sec:MainCharacter}

Let $X$ be a compact metric space and $\mathscr V = [V,p,X]$ a Hermitian line bundle. We refer the reader to \cite{Hus:fibre} for a thorough introduction to vector bundles. The set of all
continuous sections of $\mathscr V$, denoted $\Gamma (\mathscr V )$, is a right $C(X)$-module where the right action is given by multiplication. If $\{(h_U : U \times C \to \mathscr V|_U\}_{U\in 
\mathcal U}$ is an atlas for $\mathscr V$, where $\mathcal U$ is a finite open cover of $X$ and $\{\gamma_U\}_{U\in \mathcal U}$ is a partition of unity subordinate to $\mathcal U$, then

\[
\langle \xi, \eta \rangle_{\Gamma(\mathscr V)}(x)=\sum_{U \in \mathcal{U}} \gamma_U(x)\big\langle h_U^{-1}(\xi(x)), h_U^{-1}(\eta(x))\big\rangle_{\mathbb{C}^n}, 
\]
for every $\eta, \xi \in \Gamma(\mathscr V)$ and $x \in X$, makes $\Gamma(\mathscr V)$ into a right Hilbert $C(X)$-module. 
Chart maps are always assumed to preserve inner products: for each $U\in \U$, $\langle h_U(x,v),h_U(x,w)\rangle_{p^{-1}(x)}=\langle v,w\rangle_{\mathbb C}$. Thus the inner product is independent of the choice of the atlas.

Let $\alpha : X \to X$ be a homeomorphism, and define a left action of $C(X)$ on $\Gamma(\mathscr V)$ by
\[
f \xi:= \xi f\circ \alpha.
\]
The $C(X)$-valued inner product 
\[_{\Gamma(\mathscr V)} \langle \xi, \eta \rangle:=\langle \eta, \xi \rangle_{\Gamma(\mathscr V)} \circ \alpha^{-1},
\]
for every $\eta, \xi \in \Gamma(\mathscr V)$ and $f \in C(X)$ gives $\Gamma(\mathscr V)$ the structure of a Hilbert $C(X)$-bimodule, which we denote by $\Gamma(\mathscr V, \alpha)$.  

\begin{definition} The Hilbert $C(X)$-bimodule $\Gamma(\mathscr{V}, \alpha)$ is called the \emph{Hilbert $C(X)$-bimodule obtained by twisting the homeomorphism $\alpha$ by the line bundle $\mathscr V$}. The line bundle $\mathscr V$ will be referred to as \emph{the twist}.
\end{definition}

A right $A$-bimodule $\E$ is right \emph{full} if $\overline{\langle \E, \E \rangle_\E }= A$, where $\overline{\langle \E, \E \rangle_\E}$ denotes the closed linear span of $\{ \langle \xi, \eta \rangle_\E \mid \xi, \eta \in \E\}$. Left fullness is defined analogously. 
    When $A$ is left and right full, we have
\[  \E \otimes_A \E^* \cong A, \quad \text{and} \quad \E^* \otimes_A \E \cong A,\]
which is to say, $\E^*$ coincides with the dual of $\E$.

Given a compact metric space $X$, a homeomorphism $\alpha : X \to X$ and a line bundle $\mathscr{V}$ over $X$, it is straightforward to check that the Hilbert $C(X)$-module $\Gamma(\mathscr V, \alpha)$ is both left and right full. In fact, any Hilbert $C(X)$-bimodule which is both left and right full is necessarily of the form $\Gamma(\mathscr V, \alpha)$ for some line bundle $\mathscr V$ and some homeomorphism $\alpha : X \to X$ \cite[Proposition 3.7]{AAFGJSV2024}.

To construct orbit-breaking subalgebras, we require the distinguished subbimodules of $\Gamma(\mathscr V, \alpha)$ defined below.

\begin{definition} 
Set $\mathcal{E} :=\Gamma(\mathscr V, \alpha)$ and let $Y \subseteq X$ be a non-empty closed subset. The \emph{orbit-breaking submodule}, introduced in \cite{AAFGJSV2024} and denoted $\E_Y$, is defined by
\[
\E_Y:=C_0(X \setminus Y)\E,
\]
that is, $\E_Y$ is the Hilbert subbimodule of $\E$ obtained by restricting the left action to those functions $f \in C(X)$ which vanish on $Y$.
\end{definition}

We call the Cuntz--Pimsner algebra $\mathcal{O}(\E_Y)$ the \emph{orbit-breaking (sub)algebra of $\mathcal{O}(\E)$ at} $Y$.
It follows from \cite[Corollary 3.11]{AAFGJSV2024} that the Cuntz--Pimsner algebra $\mathcal{O}(\E)$ is simple if and only if $\alpha$ is a minimal homeomorphism. Moreover, if $\alpha$ is minimal, then the orbit-breaking subalgebra $\mathcal{O}(\E_Y)$ is simple if and only if $Y \cap \alpha^n(Y)=\emptyset$ for all $n \in \mathbb{Z} \setminus \{0\}$.

Orbit-breaking subalgebras $\mathcal O(\E_Y)$ are \emph{centrally large subalgebras} in the sense of \cite{ArchPhil:SR1} for closed subsets $Y$ satisfying the conditions in Theorem~\ref{centrally large} below. Centrally large subalgebras share many properties with their containing $\mathrm C^*$-algebras such as simplicity, $\mathcal Z$-stability and stable rank one. The precise definition will not be needed in this paper, so we refer the reader to \cite{ArchPhil:SR1} for the definition and further properties. In particular, the $\mathrm C^*$-algebras of  Theorem~\ref{centrally large}  are simple. It was shown in \cite{AAFGJSV2024} that if we further assume $\dim(X)<\infty$, they are  $\mathcal{Z}$-stable by \cite[Theorem 5.3]{AAFGJSV2024} and \cite[Corollary 3.5]{ArBkPh-Z}. Theorem~\ref{thm:locally sub hom} generalizes this to the case where $(X, \alpha)$ is only assumed to have the small boundary property.

\begin{theorem}\label{centrally large} \cite[Theorem 6.16]{AAFGJSV2024}
  Let $X$ be an infinite compact metric space, $\alpha : X  \to X$ a minimal homeomorphism and $\E=\Gamma(\mathscr V, \alpha)$, for a line bundle $\mathscr V = [V, p,X]$. Let $Y \subseteq X$ be a non-empty closed subset meeting each $\alpha$-orbit at most once and such that for every $N \in \mathbb{Z}_{>0}$ there exists an open set $W_N \supseteq Y$ for which $\mathscr{V}|_{\alpha^n(W_N)}$ is trivial whenever $-N \leq n \leq N$. Then $\mathcal{O}(\E_Y)$ is a centrally large subalgebra of $\mathcal{O}(\E)$.  
\end{theorem}


\section{Recursive subhomogeneity of orbit breaking subalgebras}\label{subsection:RSH}

In this section we briefly review the RSH decompositions of orbit-breaking subalgebras of \cite{FoJeSt:rsh}. 

Let $\alpha$ be a minimal homeomorphism on a compact metric space $X$ and $Y$ be a closed subset of $X$ with non-empty interior. For a point $x \in Y$, its first return time to $x$ is given by 
\[
r(x) := \min \{ n > 0 \mid \alpha^n(x) \in Y \}.
\]
Since $\alpha$ is minimal, $r(x)$ exists for every $x \in Y$. Moreover, since $X$ is compact, there are only finitely many distinct return times for points in $Y$. Let $0 < r_1 < \dots < r_K$ be the distinct values of the first return times to $Y$. For $1 \leq k \leq K$, set
\begin{equation}\label{Y-k}
  Y_k=\{x \in Y \mid r(x)=r_k\}.  
\end{equation}
 Clearly $Y$ is  the disjoint union of $Y_k$'s, that is $Y=\sqcup_{k=1}^K Y_k$.  Each $Y_k$ is not necessarily closed. But $\cup_{i=1}^k Y_i$ is closed for every $k$, hence $\overline{Y_k}\setminus Y_k=\overline{Y_k}\cap (\cup_{i=1}^{k-1}Y_i)$ is closed. One can also check that 
\begin{equation}\label{X and Y_k}
    X=\sqcup_{k=1}^K \sqcup_{i=0}^{r_k-1} \alpha^i(Y_k).
\end{equation}

Given a line bundle $\mathscr V$ over $X$, we will make use of the following line bundle 
\[\mathscr V^{(n)}=(\alpha^{n-1})^*\mathscr V\otimes \cdots \otimes \alpha^*\mathscr V\otimes \mathscr V,\ n\geq 1,\] 
where the fibre of $\mathscr V^{(n)}$ at $x$ is 
$\mathscr V^{(n)}_x=\mathscr V_{\alpha^{n-1}(x)}\otimes\cdots\otimes\mathscr V_{\alpha(x)}\otimes \mathscr V_x$. We set $\mathscr V^{(0)}:=X\times \mathbb C$ the trivial bundle. 
By \cite[Proposition 3.1]{FoJeSt:rsh}, $\Gamma(\mathscr{V}, \alpha)^{\otimes n}$ is isomorphic to $\Gamma(\mathscr{V}^{(n)}, \alpha^n)$ as Hilbert $C(X)$-bimodule.

We first recall  from \cite{FoJeSt:rsh} some ingredients which we need in subsequent discussions. 
Let $\mathscr V$ be a line bundle over $X$ and $\mathcal U$ be a finite open cover  of $X$ such that 
for each  $U\in \mathcal U$, $\mathscr V|_{U}$ is trivial  with a chart map  
\[h_{U}:U\times \mathbb C\to \mathscr{V}|_{U}.\] 
Then   the atlas
$\mathcal A=\{(U, h_{U})\mid U\in \mathcal U \}$ gives transition functions $g_{U_i,U_j}:U_i\cap U_j\to U(1)$ satisfying
\begin{equation}\label{h transition}
h_{U_i}(x,\lambda)=h_{U_j}(x,g_{U_j,U_i}(x)\lambda)   
\end{equation}
for all $x\in U_i\cap U_j$ and  $\lambda\in \mathbb C$. 
Fix $n\in \mathbb N$, and set
\[\mathcal U^{(n)}:=\{\mathbf{U}=(U_0,U_1,\dots,U_{n-1})\mid U_j\in \mathcal U \text{ and } \cap_{j=0}^{n-1}\alpha^{-j}(U_j)\neq \emptyset\}.\] 
Note that if $n=1$, then $\mathcal U^{(1)}=\mathcal U$, and that 
\[x\in \cap_{j=0}^{n-1}\alpha^{-j}(U_j)\ \Leftrightarrow \,\alpha^j(x)\in U_j\ \text{ for all } j=0, \dots, n-1.\] 
In this case, for convenience, we adopt the notation $x\in \mathbf{U}$, namely 
\[
 x\in \mathbf{U} \ \Leftrightarrow \ x\in \cap_{j=0}^{n-1}\alpha^{-j}(U_j).
\]

For each  $\mathbf{U}\in\mathcal U^{(n)}$ and $1\leq l\leq n$, 
there is  a map 
\[v^{(l)}_{\mathbf{U}}:\cap_{j=0}^{n-1}\alpha^{-j}(U_j)\to \mathscr V^{(l)}|_{\cap_{j=0}^{n-1}\alpha^{-j}(U_j)}\] given by 
\[v^{(l)}_{\mathbf{U}}(x):=h_{U_{l-1}}(\alpha^{l-1}(x),1)\otimes \cdots\otimes h_{U_1}(\alpha(x),1)\otimes h_{U_0}(x,1)\] 
which is a non-zero continuous local section of the line bundle $\mathscr V^{(l)}$ over the open set $\cap_{j=0}^{n-1}\alpha^{-j}(U_j)$. 
The set of all such open sets 
$\{\cap_{j=0}^{n-1}\alpha^{-j}(U_j)\mid \mathbf{U}=(U_0,\dots, U_{n-1})\in \mathcal U^{(n)}\}$ is a finite  cover of $X$  
and forms an atlas of $\mathscr V^{(l)}$ together with the following chart maps 
\[(x,\lambda)\mapsto \lambda v^{(l)}_{\mathbf{U}}(x):\cap_{j=0}^{n-1}\alpha^{-j}(U_j)\times \mathbb C\to \mathscr V^{(l)}|_{\cap_{j=0}^{n-1}\alpha^{-j}(U_j)}.\]

 If $x\in \mathbf{U}\in \mathcal U^{(n)}$, then  $\alpha^i(x)\in (U_i, \dots, U_{n-1})\in \mathcal U^{(n-i)}$ for $i=0,\dots, n-1$. 
 Reflecting this point we  use the following notation:
\[\alpha^i(\mathbf{U}):=(U_i, \dots, U_{n-1})\in \mathcal U^{(n-i)}.\] 
Also for $1\leq m\leq n-1$ and $0\leq i\leq n-1$ with $1\leq m+i\leq n$, we write 
\begin{equation}\label{v section}
v^{(m)}_{\alpha^i(\mathbf{U})}(\alpha^i(x)):=h_{U_{m+i-1}}(\alpha^{m+i-1}(x),1)\otimes\cdots\otimes h_{U_i}(\alpha^i(x),1)\in \mathscr V^{(m)}_{\alpha^i(x)}.
\end{equation}
Then for all $l=m+i$, $m\leq m+i\leq n$, we can write 
\begin{align*}
   &\, v^{(l)}_\mathbf{U}(x)\\
   =&\, h_{U_{l-1}}(\alpha^{l-1}(x),1)\otimes \cdots\otimes h_{U_i}(\alpha^i(x),1)\otimes  h_{U_{i-1}}(\alpha^{i-1}(x),1)\\
   &\otimes \cdots\otimes h_{U_0}(x,1)\\
   =&\,h_{U_{l-1}}(\alpha^{l-1}(x),1)\otimes \cdots\otimes h_{U_i}(\alpha^i(x),1)\otimes v_\mathbf{U}^{(i)}(x)\\
   =&\,v^{(m)}_{\alpha^i(\mathbf{U})}(\alpha^i(x))\otimes v_\mathbf{U}^{(i)}(x).
\end{align*}

Since $v_{\mathbf{U}}^{(l)}$ is a non-zero local section of $\mathscr V^{(l)}$, each unit vector $v_{\mathbf{U}}^{(l)}(x)$  serves  as a basis element of the one dimensional space $\mathscr V^{(l)}_x$ for each $x\in \cap_{j=0}^{n-1}\alpha^{-j}(U_j)$. 
In turn, one has a basis of the $n$-dimensional space  $\mathscr{D}^{(n)}_x=\mathscr{V}^{(0)}_x \oplus \mathscr{V}^{(1)}_x   \oplus \cdots \oplus \mathscr{V}^{(n-1)}_x$ at $x\in \mathbf{U}\in \mathcal U^{(n)}$ given by 
\begin{align*}
e_{\mathbf{U},0}(x) &:=(v_\mathbf{U}^{(0)}(x),\ 0,\ 0,\,\ldots,\ 0),\\
e_{\mathbf{U},1}(x) &:=(0,\ v_\mathbf{U}^{(1)}(x),\ 0,\,\dots\, ,\ 0),\\ 
   &\ \ \vdots \\
e_{\mathbf{U},n-1}(x) &:=(0,\ 0,\,\dots\,,\ 0, \ v_\mathbf{U}^{(n-1)}(x)),
\end{align*} 
with the convention that $v_{\mathbf{U}}^{(0)}(x):=1$. 
Then the map 
\[(x,(\lambda_0, \dots, \lambda_{n-1}))\mapsto \lambda_0 e_{\mathbf{U},0}(x)+\cdots+\lambda_{n-1} e_{\mathbf{U},n-1}(x)\] from $\cap_{j=0}^{n-1}\alpha^{-j}(U_j)\times\mathbb C^n$ onto $\mathscr D^{(n)}|_{\cap_{j=0}^{n-1}\alpha^{-j}(U_j)}$ is a chart map of the bundle $\mathscr D^{(n)}$ (by continuity of $e_{\mathbf{U},i}$'s over $\cap_{j=0}^{n-1}\alpha^{-j}(U_j)$) from which we obtain a chart map 
\[
H_{\mathbf{U}}^{(n)} : \cap_{j=0}^{n-1}\alpha^{-j}(U_j) \times M_n(\mathbb{C}) \to \mathscr{M}_{\mathbf{U}}^{(n)}:=\mathscr{M}^{(n)}|_{\cap_{j=0}^{n-1}\alpha^{-j}(U_j)}
\]   
of the endomorphism bundle $\mathscr{M}^{(n)}=End(\mathscr D^{(n)})$ over $\cap_{j=0}^{n-1}\alpha^{-j}(U_j)$. 
Hence, the atlas  
\[\{\big( \cap_{j=0}^{n-1}\alpha^{-j}(U_j), H_\mathbf{U}^{(n)}\big)\mid  \mathbf{U}\in \U^{(n)}\}\] of $\mathscr M^{(n)}$ gives a matrix representation of each linear map $\varsigma(x)\in \mathscr M_x^{(n)}$ for a continuous section $\varsigma$ of $\mathscr M^{(n)}$: 
For a linear map $\varsigma(x)\in \mathscr{M}^{(n)}_x$, $x\in \cap_{j=0}^{n-1}\alpha^{-j}(U_j)$, there exists a unique matrix  $\widetilde{\varsigma}_{\mathbf{U}}(x)\in M_n(\mathbb C)$ such that \[H_{\mathbf{U}}^{(n)}(x, \widetilde{\varsigma}_{\mathbf{U}}(x))  = \varsigma(x).\] 
Namely, if  $\widetilde{\varsigma}_{\mathbf{U}}(x)$ is written as 
\[
\widetilde{\varsigma}_\mathbf{U}(x) = \scriptstyle{\begin{bmatrix}
a_{00}(x) & a_{01}(x) &a_{02}(x) & \cdots & a_{0(n-1)}(x) \\
a_{10}(x) &a_{11}(x) & a_{12}(x) & \cdots & a_{1(n-1)}(x) \\
a_{20}(x) &a_{21}(x) & a_{22}(x) & \cdots & a_{2(n-1)}(x) \\
\vdots  &\vdots  & \vdots  & \ddots & \vdots  \\
a_{(n-1)0}(x) &a_{(n-1)1}(x) & a_{(n-1)2}(x)  & \cdots & a_{(n-1)(n-1)}(x) 
\end{bmatrix}},
\]
then it means that for each $j=0,\dots, n-1$, 
\[ \varsigma(x) e_{\mathbf{U},j}(x)\\
 =  a_{0j}(x)e_{\mathbf{U},0}(x) +a_{1j}(x)e_{\mathbf{U},1}(x)  + \cdots + a_{(n-1)j}(x)e_{\mathbf{U},n-1}(x).   
\]
When $x$ is fixed and clear in context, by abusing notation again, 
we write simply $\widetilde{\varsigma}_{\mathbf{U}}(x) = (H_{\mathbf{U}}^{(n)})^{-1}(\varsigma(x))$. 

Now let $x\in \mathbf{U}\cap \mathbf{V}$ for $\mathbf{U}=(U_0, \dots, U_{n-1}),\, \mathbf{V}=(V_0, \dots, V_{n-1})\in \U^{(n)}$. 
Then for $0\leq l\leq n$, one can show that 
\[v^{(l)}_{\mathbf{U}}(x)=g^{(l)}_{\mathbf{V},\mathbf{U}}(x)\cdot v^{(l)}_{\mathbf{V}}(x),\] 
where 
$g^{(l)}_{\mathbf{V},\mathbf{U}}(x):=\Pi_{j=0}^{l-1}\, g_{V_j,U_j}(\alpha^j(x)).$
With the following diagonal unitary matrix
\begin{equation}\label{def:u-UV}
    u_{\mathbf{V}\mathbf{U}}(x):=diag(1, g_{\mathbf{V},\mathbf{U}}^{(1)}(x),g_{\mathbf{V},\mathbf{U}}^{(2)}(x),\ldots,g_{\mathbf{V},\mathbf{U}}^{(n-1)}(x)),
\end{equation}
    it is easily checked that 
\begin{equation}\label{varsigma-U-V}
 \widetilde{\varsigma}_{\mathbf{V}}(x) = u_{\mathbf{V}\mathbf{U}}(x) \widetilde{\varsigma}_\mathbf{U}(x) u_{\mathbf{V}\mathbf{U}}^{-1}(x).   
\end{equation} 
In other words, 
\[H_{\mathbf{U}}^{(n)}(x, \widetilde{\varsigma}_{\mathbf{U}}(x))=H_{\mathbf{V}}^{(n)}(x, \mathrm{Ad}(u_{\mathbf{V}\mathbf{U}}(x))\widetilde{\varsigma}_{\mathbf{V}}(x)),\]
where $\mathrm{Ad}(u)$ is an automorphism such that $\mathrm{Ad}(u)M=uMu^*$ for $M\in M_n(\mathbb C)$, and   $x\mapsto \mathrm{Ad}(u_{\mathbf{V}\mathbf{U}}(x)):\cap_{j=0}^{n-1}\alpha^{-j}(U_j\cap V_j)\to \Aut(M_n(\mathbb C))$ are the transition functions for the atlas 
 $\{\big(\cap_{j=0}^{n-1}\alpha^{-j}(U_j),H_{\mathbf{U}}^{(n)}\big)\}_{\mathbf{U}\in \mathcal U^{(n)}}$ of $\mathscr M^{(n)}$.

If  $\widetilde{\varsigma}_\mathbf{U}(x)$ is written as above,  
then  with  $(x)$ omitted,
\begin{align*} 
&\ \widetilde{\varsigma}_\mathbf{V}
=\ u_{\mathbf{V}\mathbf{U}}
\widetilde{\varsigma}_\mathbf{U}
{  u_{\mathbf{V}\mathbf{U}}}^{-1}\\
=&\ \scriptstyle{
\begin{bmatrix}
a_{00} & a_{01}{g_{\mathbf{V},\mathbf{U}}^{(1)}}^{-1} & \cdots & a_{0(n-1)}{g_{\mathbf{V},\mathbf{U}}^{(n-1)}}^{-1} \\
g_{\mathbf{V},\mathbf{U}}^{(1)}a_{10} & a_{11} &  \cdots & g_{\mathbf{V},\mathbf{U}}^{(1)}a_{1(n-1)}{g_{\mathbf{V},\mathbf{U}}^{(n-1)}}^{-1} \\
\vdots  & \vdots  & \ddots  & \vdots   \\
g_{\mathbf{V},\mathbf{U}}^{(n-1)}a_{(n-1)0} &  {g_{\mathbf{V},\mathbf{U}}^{(n-1)}}a_{(n-1)1}{g_{\mathbf{V},\mathbf{U}}^{(1)}}^{-1} &  \cdots & a_{(n-1)(n-1)}
\end{bmatrix}}.
\end{align*}
From this we observe that the matrices $\widetilde{\varsigma}_\mathbf{U}(x)$ and $\widetilde{\varsigma}_\mathbf{V}(x)$ have the same diagonal elements, that is, the diagonal is invariant under the coordinate change from $\{e_{\mathbf{U},i}(x)\}$ to $\{e_{\mathbf{V},i}(x)\}$ for $x\in \mathbf{U}\cap \mathbf{V}$. Also if one of the matrices is $m$\textsuperscript{th} subdiagonal, then so is the other for $1\leq m\leq n-1$. 

If $\mathcal A'=\{(U', k_{U'})\}$ is another atlas of $\mathscr V$, then considering an atlas $\{(U\cap U', h_{U\cap U'})\}\cup \{(U\cap U', k_{U\cap U'})\}$ we can similarly show as above that for $x\in \mathbf{U}\cap \mathbf{U}'$, $\mathbf{U}\in \U^{(n)}$ and $\mathbf{U}'\in {\mathcal U'}^{(n)}$, the matrix representation $\widetilde{\varsigma}_\mathbf{U}(x)$ is $m$\textsuperscript{th} subdiagonal if and only if 
so is $\widetilde{\varsigma}_{\mathbf{U}'}(x)$ for $m=0, \dots, n-1.$

\begin{definition} \label{def:induced triv}
    We call the atlas $\{\big(\cap_{j=0}^{n-1}\alpha^{-j}(U_j),H_\mathbf{U}^{(n)}\big)\}_{U\in \mathcal U^{(n)}}$ the {\it local trivialization of $\mathscr M^{(n)}$ induced  from  the atlas}  $\{(U_j,h_{U_j})\}_{U_j\in \mathcal U}$ of $\mathscr V$. By a {\it matrix representation} of a linear map $\varsigma(x)$ in  $\mathscr{M}^{(n)}_x$, $\varsigma\in \Gamma(\mathscr M^{(n)})$, we always mean a matrix 
\[\widetilde{\varsigma}_\mathbf{U}(x)=(H_\mathbf{U}^{(n)})^{-1}(\varsigma(x))\]for some $\mathbf{U}=(U_0,\dots,U_{n-1})\in \mathcal U^{(n)}$ with $x\in \mathbf{U}$ if not mentioned otherwise.
\end{definition}

Now we return to the orbit-breaking algebras $\mathcal O(\E_Y)$ for $Y\subset X$ a closed subset with non-empty interior.  Let $r_1<\cdots<r_K$ be the first return times to $Y$ as before. 
Then the vector bundle
\[
\mathscr{D}^{(r_k)}=\mathscr{V}^{(0)} \oplus \mathscr{V}^{(1)} \oplus \mathscr{V}^{(2)} \oplus \cdots \oplus \mathscr{V}^{(r_k-1)} 
\]
over $X$ is of rank $r_k$ and its endomorphism bundle  $\mathscr{M}_{r_k}=End(\mathscr{D}^{(r_k)})$ is a locally trivial $M_{r_k}$-bundle for every $k=1, \dots, K$. 
We write  
\[\mathscr M_k:=\mathscr M_{r_k}|_{\overline{Y_k}}\] 
and  
\[H_\mathbf{U}^{(k)}:=H_\mathbf{U}^{(r_k)},\ \mathbf{U}\in \U^{(r_K)},\] 
for the rest of the paper as in \cite{FoJeSt:rsh}.
 
 \begin{lemma}\label{pi-k}{\rm (\cite[Lemma 7.3, Proposition 7.5]{FoJeSt:rsh})}
 For each $k=1,\dots, K$, there exists a homomorphism $\pi_k : \mathcal{O}(\E_Y) \to \Gamma(\mathscr M_k)$ such that for $f \in C(X)$, $\xi \in \E_Y$, and $(a_0, a_1, \dots, a_{r_k-1})\in \mathscr D^{(r_k)}_x$, 
 \begin{align*}  &\,\pi_k(f)(x)(a_0,a_1, \dots, a_{r_k-1})\\
  &=\, (f(x)a_0, f(\alpha(x))a_1, \dots,  f(\alpha^{r_k-1}(x)) a_{r_k-1}).
\end{align*} 
 and   
 \begin{align*}
  &\, \pi_k(\xi)(x)(a_0,a_1, \dots, a_{r_k-1})\\
  &=\, (0, \xi(x)\otimes a_0,  \xi(\alpha(x))\otimes a_1, \dots, \xi(\alpha^{r_k-2}(x))\otimes a_{r_k-2}). 
  \end{align*}
Moreover, $\pi:=\oplus_{k=1}^K \pi_k : \mathcal{O}(\E_Y) \to \oplus_{k=1}^K \Gamma(\mathscr{M}_k)$ is an injective homomorphism.
\end{lemma}

Suppose $x\in Y_k\cap \mathbf{U}$ for some $k$, $1\leq k\leq K$ and  $\mathbf{U}\in \U^{(r_K)}$, and consider the basis  $\{e_{\mathbf{U},0}(x),\dots, e_{\mathbf{U},r_k-1}(x)\}$ of $\mathscr D^{(r_k)}_x$.  
If $\xi\in \E_Y$, then for each $j=0,\dots,r_k-2$, 
since $\xi(\alpha^j(x))\in \mathscr V_{\alpha^j(x)}=\mathbb C\cdot h_{U_j}(\alpha^j(x),1)$, there is a  scalar $\widetilde\xi(\alpha^j(x))\in \mathbb C$ such that
\begin{equation}\label{matrix entries}
\xi(\alpha^j(x))=\widetilde\xi(\alpha^j(x))h_{U_j}(\alpha^j(x), 1).
\end{equation}
Then 
\begin{equation}\label{tau_k shift}
\pi_k(\xi)(x)e_{\mathbf{U},j}(x)=\widetilde\xi(\alpha^j(x))e_{\mathbf{U},j+1}(x)
\end{equation}
holds by Lemma~\ref{pi-k}, and for the unitary $\Lambda_x:\mathbb C^{r_k}\to \mathscr D^{(r_k)}_x$, 
\[\Lambda_x(\lambda_0, \dots, \lambda_{r_k-1})= \lambda_0 e_{\mathbf{U},0}(x)+\cdots +\lambda_{r_k-1} e_{\mathbf{U},r_k-1}(x), \] 
we have the following matrix representation 
\begin{equation}\label{tau_k matrix}
M_{\pi_k(\xi)}(x)=
\begin{bmatrix}
0 & 0 & \cdots & 0 & 0 & 0 \\
\widetilde{\xi}(x)  & 0 & \cdots & 0 & 0 & 0 \\
0 & \widetilde{\xi}(\alpha(x))  & \ddots & 0 & 0 & 0 \\
\vdots & \vdots & \ddots & \vdots & \vdots & \vdots \\
0 & 0 & \cdots & \widetilde{\xi}(\alpha^{r_k-3}(x))  & 0 & 0 \\
0 & 0 & \cdots  & 0 & \widetilde{\xi}(\alpha^{r_k-2}(x))  & 0
\end{bmatrix}
\end{equation} 
of $\Lambda_x^*\pi_k(\xi)(x)\Lambda_x\in M_{r_k}$, so 
\[\pi_k(\xi)(x)=\Lambda_x M_{\pi_k(\xi)}(x) \Lambda_x^*\ \ \text{ for } \xi\in \E_Y,\, x\in Y_k\cap \mathbf{U}. \] 
For $f\in C(X)$, the corresponding matrix representation is obviously 
\begin{equation}\label{pi_k matrix}
M_{\pi_k(f)}(x)=\diag(f(x), f(\alpha(x)), \dots, f(\alpha^{r_k-1}(x))).
\end{equation}
This matrix representation will be used in later discussion together with the fact that the matrix depends on the choice of  $\mathbf{U}\in \U^{(r_K)}$ with $x\in \mathbf{U}$, but the matrix entries differ only by multiplication by scalars of modulus one. 

\begin{remark}\label{remark:M_{r_k}}
Let $\{\gamma_U\}_{U\in \U}$ be a partition of unity subordinate to an open cover $\U$ of $X$. Define $\xi_U\in \E=\Gamma(\mathscr V,\alpha)$, $U\in \U$, by 
\[\xi_U(x)=h_U(x, \gamma_U(x)), \] 
and choose a continuous function $\Theta:X\to [0,1]$   such that $\Theta(x)=0$ if and only if $x\in Y$. Then  $\Theta\xi_U\in \E_Y$. 
If $x\in Y_k$ and $\mathbf{V}\in \U^{(r_K)}$ with $x\in \mathbf{V}$, then for $0\leq j\leq r_k-2$, by Lemma~\ref{pi-k}
\begin{align*}
    &\ \pi_k(\Theta\xi_U)(x)e_{\mathbf{V},j}(x)\\ 
    =&\ (0, \dots, \Theta(\alpha^{j+1}(x))\xi_U(\alpha^j(x))\otimes v_\mathbf{V}^{(j)}(x),\dots, 0)\\
    =&\ \Theta(\alpha^{j+1}(x))(0, \dots, h_U(\alpha^j(x),\gamma_U(\alpha^j(x))\otimes v_\mathbf{V}^{(j)}(x),\dots,0)\\
    =&\ \Theta(\alpha^{j+1}(x))\gamma_U(\alpha^j(x))(0, \dots, h_U(\alpha^j(x),1)\otimes v_\mathbf{V}^{(j)}(x),\dots,0)\\
    =&\ \Theta(\alpha^{j+1}(x))\gamma_U(\alpha^j(x))(0, \dots, h_{V_j}(\alpha^j(x), g_{V_j, U}(x))\otimes v_\mathbf{V}^{(j)}(x),\dots,0)\\
    =&\ \Theta(\alpha^{j+1}(x))\gamma_U(\alpha^j(x))g_{V_j, U}(x)(0, \dots,h_{V_j}(\alpha^j(x), 1)\otimes v_\mathbf{V}^{(j)}(x),\dots,0)\\
    =&\ \Theta(\alpha^{j+1}(x))\gamma_U(\alpha^j(x))g_{V_j, U}(x)(0, \dots,v_\mathbf{V}^{(j+1)}(x),\dots,0)\\
    =&\ \Theta(\alpha^{j+1}(x))\gamma_U(\alpha^j(x))g_{V_j, U}(x) e_{\mathbf{V},j+1}(x).
\end{align*} 
Thus the entries on the first lower subdiagonal of the matrix $M_{\pi_k(\Theta\xi_U)}(x)$, with respect to the basis $\{ e_{\mathbf V,j}\}_{j=0}^{r_k-1}$, is 
\begin{equation}\label{Theta xi_U}
\widetilde{\Theta\xi_U}(\alpha^j(x))=\Theta(\alpha^{j+1}(x))\gamma_U(\alpha^j(x))g_{V_j, U}(x).    
\end{equation} 
Note that $\widetilde{\Theta\xi_U}(\alpha^j(x))\neq 0$ exactly when $\gamma_U(\alpha^j(x))\neq 0$. 
Choose $U_j\in \U$, $0\leq j\leq r_k-1$, for which $\gamma_{U_j}(\alpha^j(x))\neq 0$. 
Considering the arguments of finitely many $g_{V_j,U_j}(x)$, one can find $c_j\in \mathbb T$ such that the first lower subdiagonal matrix \[\diag_{-1}(\lambda_0, \dots, \lambda_{r_k-2}):=\sum_{J=0}^{r_k-1} c_j M_{\pi_k(\Theta \xi_{U_j})}(x)\] 
has nonzero entries $\lambda_j$, $0\leq j\leq r_k-2$. 
It is then easy to see that $\{M_{\pi_k(\Theta\xi_U)}(x)\mid U\in \U\}$ generates the whole matrix algebra $M_{r_k}$ as $^*$-algebra. 
We will  use this fact later in the proof of Theorem~\ref{thm:FindingS}.   
\end{remark} 

The notion of boundary decomposition property of $\varsigma\in \oplus_{k=1}^K \Gamma(\mathscr M_k)$ plays an important role in determining the range of the injective homomorphism $\pi:\mathcal O(\E_Y)\to \oplus_{k=1}^K \Gamma(\mathscr M_k)$. 
Before we introduce the notion in Definition~\ref{bd property},  some notations are needed to set up first.  
Let $x\in \overline{Y_k}\setminus Y_k$, $k=1,\dots,K$, with 
\[
x \in  \overline{Y_k} \cap Y_{t_1} \cap \alpha^{-{r_{t_1}}}(Y_{t_2}) \cap \cdots
\cap \alpha^{-(r_{t_1} + r_{t_2} + \cdots + r_{t_{m(x,k)-1}})}(Y_{t_{m(x,k)}})
\] 
and $r_{t_1} + r_{t_2} + \cdots + r_{t_{m(x,k)}} =   r_{ k}$. 
Then we set 
\[
R_{x,k,0}:=0,\quad
R_{x,k,s}:=r_{t_1} + r_{t_2} + \cdots + r_{t_s}
\]
and
\[\mathscr{D}_x^{(R_{x,k,s-1}, r_{t_s})}:=\mathscr V_x^{(R_{x,k,s-1})}\oplus\mathscr V_x^{(R_{x,k,s-1}+1)}\oplus\cdots\oplus \mathscr V_x^{(R_{x,k,s-1}+r_{t_s}-1)}\] 
which is called the $s$\textsuperscript{th} component of 
\[\mathscr D^{(r_k)}_x=\mathscr D^{(r_{t_1})}_x\oplus \mathscr D_x^{(R_{x,k,1},r_{t_2})}\oplus\cdots\oplus \mathscr D_x^{(R_{x,k,m(x,k)-1},r_{t_{m(x,k)}})}.\]

\begin{definition}\label{bd property}
For each $\varsigma=(\varsigma_1,\ldots,\varsigma_K)\in \Gamma(\mathscr{M}_1) \oplus \cdots \oplus \Gamma(\mathscr{M}_K)$,  we say that $\varsigma$ has the \emph{boundary decomposition property} if it satisfies the following: for any $x\in \overline{Y_k}\setminus Y_k$, $k=1,\dots,K$, with 
\[
x \in  \overline{Y_k} \cap Y_{t_1} \cap \alpha^{-{r_{t_1}}}(Y_{t_2}) \cap \cdots
\cap \alpha^{-(r_{t_1} + r_{t_2} + \cdots + r_{t_{m(x,k)-1}})}(Y_{t_{m(x,k)}})
\] 
and $r_{t_1} + r_{t_2} + \cdots + r_{t_{m(x,k)}} =   r_{ k}$, 
the linear map $\varsigma_k(x)\in \mathscr M_k|_x=End(\mathscr D_x^{(r_k)})$ acts on each $s$\textsuperscript{th} component 
\[\mathscr{D}^{(R_{x,k,s-1}, r_{t_s})}_x=\ \mathscr D_{\alpha^{R_{x,k,s-1}}(x)}^{(r_{t_s})}\otimes  \mathscr V_x^{(R_{x,k,s-1})}\]
of $\mathscr{D}^{(r_{k})}_x$ as  
\begin{equation}\label{s-component}
\varsigma_k(x)|_{\mathscr{D}^{(R_{x,k,s-1}, r_{t_s})}_x} = \varsigma_{t_s}(\alpha^{R_{x,k,s-1}}(x)) \otimes id_{\mathscr{V}^{(R_{x,k,s-1})}_x},      
\end{equation}
or equivalently, $\varsigma_k(x)$ has its matrix representation of the following form 
\[
\widetilde{\varsigma}_k(x) = \begin{bmatrix}
\widetilde{\varsigma}_{t_1}(x) & \\
 & \widetilde{\varsigma}_{t_2}( \alpha^{R_{x,k,1}}(x)) & \\
&         & \ddots\\
&         &    & \widetilde{\varsigma}_{t_{ m(x,k)}}(\alpha^{R_{x,k,m(x,k)-1}}(x))   
\end{bmatrix}
\]
with respect to the local trivialization induced from $\{ (U,h_{U}) : U \in \mathcal{U}\}$ of  $\mathscr{V}$.	
\end{definition}

It follows from\cite[Theorem 8.6]{FoJeSt:rsh} that for $\varsigma=(\varsigma_1,\ldots,\varsigma_K)\in \Gamma(\mathscr{M}_1) \oplus \cdots \oplus \Gamma(\mathscr{M}_K)$, $\varsigma\in \pi(\mathcal O(\E_Y))$  if and only if $\varsigma$ satisfies the boundary decomposition property.

Note  that  $B_1 := \Gamma(\mathscr M_1)$  is a recursive subhomogeneous algebra by \cite[Proposition 1.7]{Phillips:recsub}, and $\pi_1:\mathcal O(\E_Y)\to  \Gamma(\mathscr M_1)$ is surjective by ~\cite[Lemma 7.8]{FoJeSt:rsh}. 
Applying the fact that $\pi(\mathcal O(\E_Y))$ is the $\mathrm{C}^*$-subalgebra of $\Gamma(\mathscr M_1)\oplus\cdots\oplus\Gamma(\mathscr M_K)$ consisting of all elements with boundary decomposition property, it is shown in \cite[Theorem 8.7]{FoJeSt:rsh} using induction on $k$ that  $B_k=\oplus_{i=1}^k \pi_i(\mathcal O(\E_Y))$ is a recursive subhomogeneous algebra. More precisely, 
\[\oplus_{i=1}^k \pi_i (\mathcal O(\E_Y))=\Gamma(\mathscr M_1)\oplus_{\Gamma(\overline{Y_2}\setminus Y_2)}\oplus \Gamma(\mathscr M_2)\oplus\cdots\oplus_{\Gamma(\mathscr M_k|_{\overline{Y_k}\setminus Y_k})} \Gamma(\mathscr M_k),\] 
where the boundary decomposition property  is used to obtain a homomorphism $\varphi_{k-1}:B_{k-1}\to \Gamma(\mathscr M_k|_{\overline{Y_k}\setminus Y_k})$. By the right hand side of the above identity, we mean the following recursive subhomogeneous algebra 
\begin{align*}
   [\cdots [\Gamma(\mathscr{M}_1)\oplus_{\Gamma(\mathscr{M}_2|_{\overline{Y}_2 \setminus Y_2})} \Gamma(\mathscr{M}_2)]\oplus \cdots ]\oplus_{\Gamma(\mathscr{M}_k|_{\overline{Y}_k \setminus Y_k})} \Gamma(\mathscr{M}_k).  
\end{align*}
Then the injectivity of $\pi$ proves that the orbit-breaking algebra $\mathcal O(\E_Y)$ is isomorphic to the following recursive subhomogeneous algebra:
\begin{align*}
  \mathcal{O}(\E_Y) \cong & \ \pi(\mathcal{O}(\E_Y))\\
  = &\ \Gamma(\mathscr M_1)\oplus_{\Gamma(\mathscr M_2|_{\overline{Y_2}\setminus Y_2})}\oplus \Gamma(\mathscr M_2)\oplus\cdots\oplus_{\Gamma(\mathscr M_K|_{\overline{Y_K}\setminus Y_K})} \Gamma(\mathscr M_K).  
\end{align*}


\section{Irreducible representations of orbit breaking subalgebras}\label{sec:irrep}

In this section we characterize all irreducible representations of orbit-breaking subalgebras via their recursive subhomogeneous structure. We end this section with a technical lemma which is essential in obtaining the main result of the paper.

Throughout the rest of this section, let $\alpha$ be a minimal homeomorphism over a compact metric space $X$, $\mathscr V=[V,p,X]$ be a line bundle and $Y$ be a closed subset of $X$ with non-empty interior. Let 
\[r_1 < \dots < r_K\]
denote the distinct first return times to $Y$.

Let $\U$ be a finite open covering of $X$ for which  $\mathcal A=\{(U,h_U)\mid U\in \U\}$ forms an atlas of the line bundle $\mathscr V$. 
By $\U^{(r_K)}$, we denote the set of all ordered sequences $\mathbf{U}:=(U_0, \dots, U_{r_K-1})$ of open sets $U_j\in \U$ such that $\cap_{j=0}^{r_K-1}\alpha^{-j}(U_j)\neq \emptyset$ as before.
Then the open sets $\{\cap_{j=0}^{r_K-1}\alpha^{-j}(U_j)\mid \mathbf{U}\in\U^{(r_K)}\}$ form a finite covering of $X$ and for each $\mathbf{U}\in \U^{(r_K)}$, as in Section~\ref{subsection:RSH}, we obtain a chart map
\[H_\mathbf{U}^{(k)}: \cap_{j=0}^{r_K-1}\alpha^{-j}(U_j)\cap \overline{Y_k}\times M_{r_k}\to \mathscr M_k|_{\cap_{j=0}^{r_K-1}\alpha^{-j}(U_j)\cap\overline{Y_k}}\] 
for $1\leq k\leq K$.

\begin{remarks}\label{rep of A_Y}
\begin{enumerate}[leftmargin=*] 
\item If $\mathscr M$ is a locally trivial bundle over a second-countable locally compact Hausdorff space $T$ with fibre $M_n$, then  the spectrum $\widehat{\Gamma}_0(\mathscr M)$ of the $\mathrm{C}^*$-algebra $\Gamma_0(\mathscr M)$ of all continuous sections vanishing at infinity is homeomorphic to the base space $T$ since the fibre $M_n$ is simple (for example, see \cite[Theorem 1.1 and Corollary of Theorem 1.2]{Fell:OpFields}). 
Actually, the homeomorphism is given by 
\[x\mapsto [\nu_x]: T\to \widehat{\Gamma}_0(\mathscr M),\] 
 where $[\nu_x]$ denotes the unitary equivalence class of an irreducible representation $\nu_x: \Gamma(\mathscr M)\to M_n$, 
\[\nu_x(\varsigma):=H(\varsigma(x)), \ \varsigma\in \Gamma_0(\mathscr M)\] for an isomorphism $H:\mathscr M|_x \to M_n$.
 \item In our case, by (1) the spectrum of $\Gamma(\mathscr M_k)$ is homeomorphic to the base space  $\overline{Y_k}$ of $\mathscr M_k$ for each $k=1, \dots, K$, and every irreducible representation of $\Gamma(\mathscr M_k)$ is (unitarily) equivalent to the following $r_k$ dimensional irreducible representation 
\[\varsigma\mapsto (H_\mathbf{U}^{(k)})^{-1}(\varsigma(x)):\Gamma(\mathscr M_k)\to M_{r_k}\] 
for some $x\in \overline{Y_k}$ and $\mathbf{U}\in \U^{(r_K)}$ with $x\in \overline{Y_k}\cap \mathbf{U}$. 
Thus,  the spectrum of  $\Gamma(\mathscr M_1)\oplus\dots\oplus \Gamma(\mathscr M_K)$ is homeomorphic to $\sqcup_{k=1}^ K  \overline{Y_k}$ the disjoint union of the base spaces $\overline{Y_k}$.
Namely, every irreducible representation of $\Gamma(\mathscr M_1)\oplus\dots\oplus \Gamma(\mathscr M_K)$ is equivalent to an irreducible representation 
\[\nu_{x,k,\mathbf{U},\mathcal A}: \Gamma(\mathscr M_1)\oplus\cdots\oplus \Gamma(\mathscr M_K)\to M_{r_k}\]
given by \[\nu_{x,k,\mathbf{U},\mathcal A}(\varsigma_1, \dots, \varsigma_K):=(H_\mathbf{U}^{(k)})^{-1}(\varsigma_k(x))\] 
for some $k$, $1\leq k\leq K$, $x\in \overline{Y_k}$ and $\mathbf{U} \in \U^{(r_K)}$ with $x\in \mathbf{U}$.
\item Note that $\overline{Y_k}\cap\overline{Y_l}\neq \emptyset$ is possible for $k\neq l$. If $x\in \overline{Y_k}\cap \overline{Y_l}$ for some $k,l$ with $1\leq l<k\leq K$, then $x\in (\overline{Y_k}\setminus Y_k)\cap Y_{t_1}$ for some $t_1$ with $1\leq {t_1}\leq l<k\leq K$, hence 
\[\nu_{x,k,\mathbf{U},\mathcal A}(\varsigma)\in M_{r_k}\, \text{ while  }\,\nu_{x,{t_1}\mathbf{U},\mathcal A}(\varsigma)\in M_{r_{{t_1}}},\] 
for all $\varsigma\in \Gamma(\mathscr M_1)\oplus\cdots\oplus \Gamma(\mathscr M_K)$ and so these two irreducible representations 
$\nu_{x,k,\mathbf{U},\mathcal A}$ and $\nu_{x,t_1,\mathbf{U},\mathcal A}$  are not equivalent. 
Moreover, we will see in Proposition~\ref{irr rep of A_Y} and its proof that $\nu_{x,t_1,\mathbf{U},\mathcal A}|_{\pi(\mathcal O(\E_Y))}$ is irreducible while $\nu_{x,k,\mathbf{U},\mathcal A}|_{\pi(\mathcal O(\E_Y))}$ is not.
\end{enumerate}
\end{remarks}

Set  
$B_Y:= \Gamma(\mathscr M_1)\oplus\cdots\oplus \Gamma(\mathscr M_K)$ and 
$A_Y:=\pi(\mathcal O(\E_Y)$). Then
\[A_Y = \Gamma(\mathscr M_1)  \oplus_{\Gamma(\mathscr M_2|_{\overline{Y_2}\setminus Y_2})}\oplus  \Gamma(\mathscr M_2)\oplus\cdots\oplus_{\Gamma(\mathscr M_K|_{\overline{Y_K}\setminus Y_K}) } \Gamma(\mathscr M_K)\] 
is a $\mathrm{C}^*$-subalgebra of $B_Y$ such that every irreducible representation of $A_Y$ is the restriction of an irreducible representation of $B_Y$ up to unitary equivalence as we show in the following proposition.

\begin{proposition}\label{irr rep of A_Y} 
Every irreducible representation $\nu_{x,k,\mathbf{U},\mathcal A}$, $x\in Y_k\cap \mathbf{U}$, $\mathbf{U}\in \U^{(r_K)}$,  of $B_Y$ restricts to an irreducible representation of $A_Y$.

Conversely, if $\rho$ is an irreducible representation of $A_Y$, then $\rho$ is equivalent to $\nu_{x,k,\mathbf{U},\mathcal A}|_{A_Y}$ for some $k$, $1\leq k\leq K$, $x\in Y_k$, and $\mathbf{U}\in \mathcal U^{(r_K)}$ with $x\in Y_k\cap \mathbf{U}$.    
\end{proposition}
\begin{proof}
To show that $\nu_{x,k,\mathbf{U},\mathcal A}|_{A_Y}$, $x\in Y_k\cap \mathbf{U}$, is an irreducible representation of $A_Y$, we claim that $\nu_{x,k,\mathbf{U},\mathcal A}(A_Y)= \nu_{x,k,\mathbf{U},\mathcal A}(B_Y)$. 
For this, let $(\varsigma_1, \dots, \varsigma_K)\in B_Y$. Since $x\in Y_k$, there is an open neighborhood $V$ of $x$ such that $V\cap (\overline{Y_k}\setminus Y_k)=\emptyset$ (since $\overline{Y_k}\setminus Y_k$ is closed in $X$). 
Choose a continuous function $g\in C_0(V)\subset C(X)$ with $g(x)=1$. Then $\varsigma_k':=\varsigma_k \cdot g|_{\overline Y_k}\in \Gamma(\mathscr M_k)$ satisfies $\varsigma_k'(x)=\varsigma_k(x)\in H_{\mathbf U}^{(k)}(x,M_{r_k})$. 
Since $\varsigma_k'|_{\overline{Y_k}\setminus Y_k}=0$, we have 
\[(0, \dots, 0, \varsigma_k')\in B_k:=\Gamma(\mathscr M_1)\oplus_{\Gamma(\overline{Y_2}\setminus Y_2)}\Gamma(\mathscr M_2)\oplus\cdots\oplus_{\Gamma(\overline{Y_k}\setminus Y_k)}\Gamma(\mathscr M_k)\] 
as in \cite[Theorem 8.7]{FoJeSt:rsh} and by the same theorem, there is an element $a\in \mathcal O(\E_Y)$ such that 
$(0, \dots, 0, \varsigma_k')=\oplus_{i=1}^k \pi_i(a)$. 
Then $\pi(a)\in A_Y$ satisfies 
\begin{align*}
\nu_{x,k,\mathbf{U},\mathcal A} (\pi(a))=&\ (H_\mathbf{U}^{(k)})^{-1}(\varsigma_k'(x))\\ 
= &\ (H_\mathbf{U}^{(k)})^{-1}(\varsigma_k(x))\\
= &\ \nu_{x,k,\mathbf{U},\mathcal A}(\varsigma_1,\dots,\varsigma_K).  
\end{align*}

 For the converse, let $\rho:A_Y\to B(\mathcal H_\rho)$ be an irreducible representation. Since $A_Y$ is a $\mathrm C^*$-subalgebra of $B_Y$, $\rho$  extends to an irreducible representation $\widetilde\rho : B_Y\to B(\mathcal H_{\widetilde \rho})$ on a Hilbert space $\mathcal H_{\widetilde\rho}$ which contains $\mathcal H_\rho$ as a closed subspace and $\rho(\varsigma)=\widetilde\rho(\varsigma)|_{\mathcal H_\rho}$ for $\varsigma\in A_Y$. 
By Remarks~\ref{rep of A_Y}(2), $\widetilde\rho$ is equivalent to an irreducible representation $\nu_{x,k,\mathbf{U},\mathcal A}$ of $B_Y$   
for some $k$, $\mathbf{U}\in \mathcal U^{(r_K)}$ and 
$x\in \overline{Y_k}\cap \mathbf{U}$ given by 
 \[ \nu_{x,k,\mathbf{U},\mathcal A}(\varsigma_1, \dots, \varsigma_K)=(H_\mathbf{U}^{(k)})^{-1}(\varsigma_k(x))\in M_{r_k},\] 
for $(\varsigma_1, \dots, \varsigma_K)\in B_Y$.
 Thus $\mathcal H_{\widetilde\rho}\cong\mathbb C^{r_k}$. 
 Note that  $x\in \overline{Y_k}$ is either $x\in Y_k$ or $x\in \overline{Y_k}\setminus Y_k$. 

 If $x\in Y_k$, then 
 \[
  \widetilde\rho (A_Y)=M_{r_k}= \widetilde\rho (B_Y)
 \]
 holds as shown in the first part of the proof. Hence $\widetilde\rho|_{A_Y}$ is an irreducible representation of $A_Y$, so there is no nontrivial closed subspace of $\mathcal H_{\widetilde\rho}\cong \mathbb C^{r_k}$ which is invariant under $\widetilde\rho|_{A_Y}$, which proves that $\mathcal H_\rho=\mathcal H_{\widetilde\rho}$, and $\rho$ is equivalent to the restriction  $\nu_{x,k,\mathbf{U},\mathcal A}|_{A_Y}$.

If $x\in \overline{Y_k}\setminus Y_k$, then 
\[x\in \overline{Y_k}\cap Y_{t_1}\cap \alpha^{-r_{t_1}}(Y_{t_2})\cap\cdots\cap \alpha^{-(r_{t_1}+\cdots +r_{t_{m-1}})}(Y_{t_m})\] 
for some $t_i$ with
$r_{t_1}+\cdots + r_{t_m}=r_k$, $m:=m(x,k)$, as before. 
Then the matrix representation of $\nu_{x,k,\mathbf{U},\mathcal A}(\varsigma_1, \dots, \varsigma_K)$ for $(\varsigma_1, \dots, \varsigma_K)\in A_Y$ is of the following block diagonal form 
\[
\widetilde{\varsigma}_k(x) = \begin{bmatrix}
\widetilde{\varsigma}_{t_1}(x) & \\
 & \widetilde{\varsigma}_{t_2}( \alpha^{R_{x,k,1}}(x)) & \\
&         & \ddots\\
&         &    & \widetilde{\varsigma}_{t_{m}}(\alpha^{R_{x,k,m-1}}(x))   
\end{bmatrix} 
\]
due to the boundary decomposition property of the elements in $A_Y$. 
So obviously the restriction $\widetilde\rho|_{A_Y}$, equivalent to $\nu_{x,k,\mathbf{U},\mathcal A}|_{A_Y}$, is not irreducible. From this matrix representation, we also see that the irreducible representation $\rho: A_Y\to B(\mathcal H_\rho)$, $\mathcal H_\rho\leq \mathcal H_{\widetilde \rho}\cong \mathbb C^{r_k}$, which is a subrepresentation of $\widetilde\rho|_{A_Y}$, must be  equivalent to an irreducible representation having the following matrix representation 
\[(\varsigma_1, \dots, \varsigma_K)\mapsto \Bigg[ \widetilde \varsigma_{t_i}(\alpha^{R_{x,k,i}}(x))\Bigg]=(H_{\mathbf{U}'}^{(t_i)})^{-1}(\varsigma_{t_i}(\alpha^{R_{x,k,i}}(x)))\in M_{r_{t_i}}\]
for some $i$, $0\leq i\leq m-1$. In other words, $\rho$ is equivalent to $\nu_{x',t_i, \mathbf{U}',\mathcal A}|_{A_Y}$ for some $x'=\alpha^{R_{x,k,i}}(x)\in Y_{t_i}$ and $\mathbf{U}'\in \U^{(r_K)}$ with $x'\in \mathbf{U}'$, as desired. 
\end{proof}

Since $A_Y=\pi(\mathcal O(\E_Y))$ is a subhomogeneous algebra,  $[\rho]\mapsto \ker(\rho)$ is a homeomorphism of the spectrum $\widehat{A_Y}$ of $A_Y$ onto  the primitive ideal space $\mathrm{Prim}(A_Y)$ of $A_Y$, hence we may identify each class $[\nu_{x,k,\mathbf{U},\mathcal A}|_{A_Y}]$ with its kernel, $\ker(\nu_{x,k,\mathbf{U},\mathcal A}|_{A_Y})$. 

If $\mathcal A'$ is another atlas of $\mathscr V$ and $x\in Y_k\cap \mathbf{U}\cap \mathbf{U}'$ for $\mathbf{U}\in \U^{(r_K)}$ and $\mathbf{U}'\in {\mathcal U'}^{(r_K)}$, then $\nu_{x,k,\mathbf{U},\mathcal A}|_{A_Y}$ is equivalent to $\nu_{x,k,\mathbf{U}',\mathcal A'}|_{A_Y}$, that is, 
$[v_{x,k,\mathbf{U},\mathcal A}|_{A_Y}]=[v_{x,k,\mathbf{U}',\mathcal A'}|_{A_Y}]$. 
Namely, the class depends only on $x$ and $k$ with $x\in Y_k$, which justify the following notation 
\[[\nu_{x,k}]:=[v_{x,k,\mathbf{U},\mathcal A}|_{A_Y}]\] 
even without any exact definition of $\nu_{x,k}$, and we may write the corresponding primitive ideal of $A_Y$ as $\ker(\nu_{x,k})$.

The following corollary is \cite[Lemma 2.1]{Phillips:recsub} in case every $\mathscr M_k$ is a trivial $M_{r_k}$-bundle, so that $\Gamma(\mathscr 
 M_k)\cong C(\overline{Y_k}, M_{r_k})$. 

\begin{corollary}\label{prim A_Y}
  The spectrum $\widehat{A_Y}$ of 
  \[A_Y=\Gamma(\mathscr M_1)\oplus_{\Gamma(\mathscr M_2|_{\overline{Y_2}\setminus Y_2})}\oplus\cdots\oplus_{\Gamma(\mathscr M_K|_{\overline{Y_K}\setminus Y_K})}\Gamma(\mathscr M_K)\]  is homeomorphic to the disjoint union $Y_1\sqcup\cdots\sqcup Y_K$ via the map 
  \[x\mapsto [\nu_{x,k}]: Y_k\to \widehat{A_Y}\] 
  for each $k=1, \dots, K$.
\end{corollary} 

\begin{proof}  
By Proposition~\ref{irr rep of A_Y}, the maps $x\mapsto [\nu_{x,k}]$, $x\in Y_k$, $k=1, \dots, K$, define a bijection of  $\sqcup_k Y_k$ onto the spectrum $\widehat{A}_Y$. We show that this bijection is in fact a homeomorphism.

We fix $x\in Y_k$ and show that for a sequence  $\{x_n\}$ in $Y_k$,  
\[x_n\to x\ \text{ if and only if }[\nu_{x_n,k}]\to [\nu_{x,k}].\] 
For this, first note that $[\nu_{x_n,k}]\to [\nu_{x,k}]$ means that $\ker(\nu_{x,k})$ is a limit point of every subsequence  of $\{\ker(\nu_{x_{n},k})\}_n$, namely 
\[\cap_{i\geq 1}\ker(\nu_{x_{n_i},k}) \subset \ker(\nu_{x,k})\] 
for every subsequence $\{x_{n_i}\}_i$ of $\{x_n\}_n$.

Let $x_n\to x$ and $\{x_{n_i}\}_i$ be a subsequence of $\{x_n\}$. Then for any $\mathbf{U}\in \U^{(r_K)}$ with $x\in \mathbf{U}$, one can choose $i_0$ such that $x_{i}\in \mathbf{U}$ for all $i\geq i_0$. 
If $\varsigma=(\varsigma_1, \dots,\varsigma_K)\in \ker(\nu_{x_{n_i},k})$ for all $i\geq i_0$, then  
\[\nu_{x_{n_i},k,\mathbf{U},\mathscr A}|_{A_Y}(\varsigma)=(H_{\mathbf{U}}^{(k)})^{-1}(\varsigma_k(x_{n_i}))=0 \] for all $i\geq i_0$, and  
the continuity of $\varsigma_k$ at $x$ gives 
$(H_{\mathbf U}^{(k)})^{-1}(\varsigma_k(x))=0$. 
This means $\varsigma\in \ker(\nu_{x,k})$, hence 
$\cap_{i\geq 1}\ker(\nu_{x_{n_i},k})\subset \ker(\nu_{x,k})$ follows from $\cap_{i\geq 1}\ker(\nu_{x_{n_i},k})\subset \cap_{i\geq i_0}\ker(\nu_{x_{n_i},k})$.
Thus we have $[\nu_{x_n,k}]\to [\nu_{x,k}]$.

Now let $[\nu_{x_n,k}]\to [\nu_{x,k}]$. 
Suppose $x_n\nrightarrow x$. Passing to a  subsequence of $\{x_n\}$, we may assume that there is an open neighborhood $U$ of $x$ such that $x_n\notin U$ for all $n$. We may also assume that $U\cap (\cup_{l=1}^{k-1}Y_l)=\emptyset$ since $x$ is not in the closed set $\cup_{l=1}^{k-1}Y_l$. 
Then one can find a $\varsigma_k\in \Gamma(\mathscr M_k)$ such that $\varsigma_k(x)\neq 0$ and $\supp(\varsigma_k)\subset U$ (hence $\varsigma_k(x_n)=0$ for all  $n\geq 1$). 
The condition $U\cap (\cup_{l=1}^{k-1}Y_l)=\emptyset$ implies that $\varsigma_k|_{\overline{Y_k}\setminus Y_k}=0$, and hence 
\[(0,\dots,0,\varsigma_k)\in \Gamma(\mathscr M_1)\oplus_{\Gamma(\mathscr M_2|_{\overline{Y_2}\setminus Y_2})}\oplus\Gamma(\mathscr M_2)\oplus\dots\oplus_{\Gamma(\mathscr M_{k}|_{\overline{Y_k}\setminus Y_k})}\Gamma(\mathscr M_k).\]
For any $a\in \mathcal O(\E_Y)$ with $\oplus_{j=1}^k \pi_j(a)=(0,\dots, 0,\varsigma_k)$,  the element 
$\varsigma:=\pi(a)\in A_Y$ satisfies 
\[\varsigma\notin \ker(\nu_{x,k})\ \text{and } \varsigma\in \ker(\nu_{x_n,k})\ \text{for all } n\geq 1,\] 
which contradicts the assumption $[\nu_{x_n,k}]\to [\nu_{x,k}]$.
\end{proof}

If $\mathscr M$ is a locally trivial $M_n$-bundle over a locally compact Hausdorff space $T$, then by Remarks~\ref{rep of A_Y}(1) the spectrum $\widehat{\Gamma_0}(\mathscr M)$ of $\Gamma_0(\mathscr M)$ (hence the primitive ideal space $\text{Prim}(\Gamma_0(\mathscr M))$)  is  homeomorphic to the base space $T$ of $\mathscr M$. 
Let $S\subset A:=\Gamma_0(\mathscr{M})$ be a $C^*$-subalgebra such that $\nu|_S$ is irreducible for any irreducible representation $\nu$ of $A$. Then $S$ is subhomogeneous and $\widehat S=\{[\nu_x|_S]\mid x\in T\}$, where $\nu_x$ is an irreducible representation of $\Gamma_0(\mathscr M)$ with kernel corresponding to the point $x\in T$. 
 The relation $\sim$ on $T=\widehat A$ defined by 
 \[x\sim y \ \text{ if and only if } \ [\nu_x|_S]=[\nu_y|_S]\] is an equivalence relation and $\widehat S$ is homeomorphic to the quotient space $\widehat A/\sim$ equipped with the quotient topology induced from $\widehat A=T$ (see \cite[Lemma 1.10]{Fell:OpFields}).  
 Note here that $[\nu_x|_S]\in \widehat S$ is in the closure of a subset $\{[\nu_{x_i}|_S]\}_i\subset \widehat S$  exactly when $\ker(\nu_x|_S)\supset \cap_i \ker(\nu_{x_i}|_S)$, namely when $\nu_x(a)=0$ for any $a\in S$ with $\nu_{x_i}(a)=0$ for all $i$.

The following is a slightly more general version of \cite[Lemma 4.3]{EllNiu:MeanDimZero}.

\begin{lemma} \label{lemma:SW}
Let $\mathscr{M} = [M, p, T]$ be a locally trivial $M_n$-bundle over a second-countable locally compact Hausdorff space $T$ and let $A = \Gamma_0(\mathscr{M})$ the $\mathrm{C}^*$-algebra of continuous sections vanishing at infinity. Suppose that $S \subset A$ is a $\mathrm C^*$-subalgebra and there exists a topological space $\Delta$ and a surjective continuous map
\[\sigma : T \to \Delta\] 
such that
\begin{enumerate}
    \item[\rm{(1)}] for any $x_1, x_2$ we have that $\sigma(x_1) = \sigma(x_2)$ if and only if $[\nu_{x_1}|_S]=[\nu_{x_2}|_S]$ ;
    \item[\rm{(2)}] for any sequence $(x_i)_{i \in \mathbb{N}} \subset T$, any $x \in T$ such that $\sigma(x_i) \to \sigma(x)$, we have $a(x)=0$ whenever $a\in S$ satisfies $a(x_i)=0$ for all $i$ ;   
    \item[\rm{(3)}] $\nu_x(S) \cong M_n$ for any $x \in T$.
\end{enumerate}
Then $S$ is $n$-homogeneous with $\prim(S)$  homeomorphic to $\Delta$.
\end{lemma}

\begin{proof}  It follows from (3) that every irreducible representation of $A$ restricts to an irreducible representation of $S$. 
So we can identify $\widehat S$ and $\widehat  A/\sim$ as explained before this lemma. 
Since the quotient topology of $\widehat  S=\widehat A/\sim$  is the finest one for which the quotient map $q:\widehat A\to \widehat A/\sim$, $q([\nu_x])=[\nu_x|_S]$, is continuous, condition (1) and the continuity of $\sigma$ shows that the bijective map, say 
\[\tilde\sigma:\widehat A/\sim\ \to\, \Delta,\ \ {\tilde\sigma}([\nu_x|_S]):=\sigma(x),\] is continuous. 
In fact, if $U\subset \Delta$ is an open subset, then $q^{-1}({\tilde\sigma}^{-1}(U))=\sigma^{-1}(U)$ is open in $\widehat A=T$ which means ${\tilde\sigma}^{-1}(U)$ is open in $\widehat A/\sim$ with respect to the quotient topology. The continuity of the inverse of $\tilde\sigma$ follows from (2), and we conclude that $\prim(S)=\widehat S=\widehat A/\sim$ is homeomorphic to $\Delta$. 
\end{proof}


\section{\texorpdfstring{$\mathcal Z$-stability}{Jiang--Siu stability}}\label{sec:Zstab}

In this section we apply the structural results developed in the previous sections to the setting of RSH $\mathrm C^*$-algebras arising from minimal dynamical systems twisted by line bundles. Orbit capacity estimates provided by the small boundary property allow for uniform control over the building blocks in the RSH model. This allows us to obtain precise bounds on dimension growth and ultimately verify $\mathcal Z$-stability.  The notion of dimension growth required is given by the \emph{dimension ratio}.

\begin{definition} \label{def:dimrat}
Let $A$ be subhomogeneous with dimensions of irreducible representations $d_1 < d_2 < \dots < d_k$. For $1 \leq i \leq k$, let  $\prim_{d_i}(A)$ denote the space of primitive ideals corresponding to irreducible representations of dimension $d_i$, equipped with the relative Jacobson topology.  The \emph{dimension ratio} of $A$ is then defined by
 \[ \dim\textup{Ratio} (A) := \max_{1 \leq i \leq k} \left \{ \frac{\dim(\prim_{d_i}(A))}{d_i} \right \}.\]
\end{definition}

The next theorem, while rather technical, shows that we are able to approximate  orbit-breaking subalgebras by subhomogeneous $\mathrm C^*$-algebras with arbitrarily small dimension growth. This is precisely what is needed to deduce $\mathcal Z$-stability.

\begin{theorem} \label{thm:FindingS} Let $X$ be an infinite compact metric space, $\mathscr{V}$ a line bundle over $X$ and $\alpha : X \to X$ a minimal homeomorphism with the small boundary property. Let $\E = \Gamma(\mathscr{V}, \alpha)$. Fix $y \in X$ and  $W \subset X$ an open neighborhood of $y$. Suppose $\F_1 = \{ f_1, \dots, f_m\} \subset C(X)$, $\F_2 = \{g_{1}, \dots, g_{m} \} \subset C_0(X \setminus W)$,  $\{\xi_1,\dots,\xi_m\mid \xi_i\in \E,\ \|\xi_i\|\leq 1\}$, and $\epsilon > 0$ are given. Then there exists $Y \subset W$  a closed neighborhood of $y$ and a subhomogeneous $C^*$-algebra $S \subset \mathcal{O}(\E_Y)$ such that 
\begin{enumerate}
\item[\rm{(1)}] $\dim\textup{Ratio} (S) < \epsilon$, 
\item[\rm{(2)}] $\dist(f_i, S) < \epsilon$, for every $1 \leq i \leq m$, and
\item[\rm{(3)}] $\dist(g_{i} \xi_i, S) < \epsilon$, for every $1 \leq i \leq m$.
\end{enumerate}
\end{theorem}

\begin{proof} 
If  $(U,h_U)$ is a chart of $\mathscr V$,  then the section $\xi_i|_U$ induces a continuous function $s_{i,U}\in C(X)$  such that 
\[\xi_i(x)=h_U (x, s_{i,U}(x)) \ \text{for }\, x \in U.\]
Let 
\[
\epsilon_0 < \epsilon/6.
\]
Fix a finite open cover $\mathcal{U}$ of $X$ such that $\mathscr V|_U$, $U\in \U$, is trivial with chart map $h_U:U\times \mathbb C\to \mathscr V|_U$  and the following hold:
\begin{enumerate}[label=(\alph*)]
  \item   $|f_i(x) - f_i(x') | < \epsilon_0$, 
  \item \( |g_i(\alpha(x))s_{i,U}(x)-g_i(\alpha(x'))s_{i,U}(x')| < \epsilon_0 \),
\end{enumerate}
for every $x, x' \in U$, $U \in \mathcal{U}$, and  $1 \leq i \leq m$. For $U,U'\in \U$, $g_{U,U'}:U\cap U'\to \mathbb U(1)$ denotes the transition function satisfying
\[h_U(x,\lambda)=h_V(x, g_{U',U}(x)\lambda),\ \lambda\in \mathbb C,\  x\in U\cap U'.\]
We put $ g_{U,U'}=0$ if $U \cap U' =\emptyset$ for $U,U' \in \U$.

By Proposition~\ref{prop:SBP-POU}, there exists a partition of unity $\{\gamma_U\}_{U \in \mathcal{U}}$ subordinate to $\mathcal{U}$ and a $T \in \mathbb{N}$ such that
    \[ \frac{1}{N} \sum_{j=0}^{N-1} \chi_{\cup_{U \in\mathcal{U}} \gamma_U^{-1}((0,1))}(\alpha^j(x)) < \frac{\epsilon_0}{|\mathcal{U}|(|\U|+1)}, \text{ for every } x \in X, N \geq T. \]

Using minimality of $\alpha$, the orbit of $y$ is dense in $X$, so there is a closed neighborhood $Y \subset W$ of $y$ with  first return times $ 0<r_1 < r_2 < \dots < r_K$ satisfying 
\[ r_1 > \max \{ \frac{|\mathcal{U}|(|\U|+1)}{\epsilon_0}, T\}.\]

 As in the proof of \cite[Theorem 4.5]{EllNiu:MeanDimZero} (which uses \cite[Lemma 4.1]{EllNiu:MeanDimZero}),  since $Y$ has non-empty interior, there is an open subset $W_0 \subset W$ containing $Y$ such that, for any $1 \leq k \leq K$ we have
\[  r_k^{-1} \sum_{j=0}^{r_k-1} \chi_{{W_0}}(\alpha^j(x)) \leq r_1^{-1} < {\frac{\epsilon_0}{|\mathcal{U}|(|\U| + 1)}}, \text{ for every } x \in  Y_k.\]  

Let $\Theta : X \to [0,1]$ be a function satisfying $\Theta(x) = 0$ if and only if $x \in Y$ and $\Theta|_{X \setminus  W_0} = 1$. Note from $W_0 \subset W$  that $\Theta g_i = g_i$ for every $1 \leq i \leq m$. For each $U\in \U$, define $\xi_U\in \E$  by  
\[\xi_U(x):=h_U(x, \gamma_U(x)).\]
Then $\Theta\xi_U\in\E_Y$ for every $U\in \mathcal U$.  
Now we consider the $C^*$-subalgebra 
\[S:=C^*(\{\gamma_U\}_{U\in \mathcal U},\, \{\Theta\xi_U\}_{U\in \mathcal U})\subset \mathcal O(\E_Y).\]
Note that $S$ is the subalgebra of a subhomogeneous $\mathrm C^*$-algebra, hence is subhomogeneous. 

Define $\gamma_\U$, $\Psi :X\to \mathbb R^{|\U|}$ by 
\[\gamma_\U(x):=\bigoplus_{U\in \U} \gamma_U(x) \ \ \text{ and } \ \Psi(x)=\bigoplus_{U \in \mathcal{U}} \Theta(\alpha(x))\gamma_U(x).\] 
(Note that $h_U(x,\Theta(\alpha(x))\gamma_U(x))=\Theta\xi_U(x)$.) 
Let $\mathbb D:=\{z\in \mathbb C\mid |z|\leq 1\}$ and define   $\Phi : X \to \mathbb{D}^{|\U|(|\U|-1)}$ by
\[  \Phi(x) := \bigoplus_{U\neq V \in \mathcal{U}} \gamma_U(x)\gamma_V(x) g_{UV}(x).\] 
For every $1 \leq k \leq K$, let
 \[ \sigma_k : \overline{Y_k} \to \mathbb{R}^{|\mathcal{U}|(2r_k-1)} \oplus \mathbb{D}^{|\U|(|\U|-1) r_k}\] 
 be the map given by
 \begin{align*} 
 \sigma_{k}(x) =& (\gamma_\U(x), \gamma_\U \circ\alpha(x), \dots, \gamma_\U\circ \alpha^{r_k-1}(x), \\ 
  &\ \ \ \ \Psi(x), \Psi \circ\alpha(x), \dots, \Psi\circ \alpha^{r_k-2}(x), \\
 & \ \ \ \ \ \ \ \Phi(x), \Phi\circ\alpha(x), \dots, \Phi\circ\alpha^{r_k-1}(x)). 
 \end{align*} 
 Let $x\in Y_k$ and let 
 \[\sigma_k(x)=(t_1,\dots, t_{|\U|r_k}, \dots, t_{|\mathcal{U}|(2r_k-1)}, s_1, \dots, s_{|\U|(|\U|-1)r_k}).\]
Then among the first $|\U|r_k$ coordinates of $\sigma_k(x)$, at most $\epsilon_0 r_k$ of $t_i$'s are not equal to $0$ and $1$. In fact, since $t_i=\gamma_U(\alpha^j(x))$ for some $U\in \U$ and $0\leq j\leq r_k-1$, 
\begin{align*}
 &\ |\{(U,j)\mid \gamma_U(\alpha^j(x))\in (0,1)\}|  \\
 \leq &\ |\U|\, |\{0\leq j\leq r_k-1\mid \alpha^j(x)\in \cup_{U\in \U} \gamma_U^{-1}(0,1)\}|\\
 \leq &\ |\U|\frac{\epsilon_0}{|\U|(|\U|+1)}r_k\leq \frac{\epsilon}{|\U|+1} r_k\leq \epsilon_0 r_k.
\end{align*} 
The coordinates $t_i$ for $i=|\U|r_k+1,\dots, |\mathcal{U}|(2r_k-1)$, are of the form  $t_i=\Theta(\alpha^{j+1}(x))\gamma_U(\alpha^j(x))$ for some $U\in \U$ and $0\leq j\leq r_k-2$. In this case,  note that $\Theta(\alpha^{j+1}(x))\neq 0$ and $\gamma_U(\alpha^j(x))=1$ is possible only for a single $U$. We thus have 
\begin{align*}
&\  |\{  (U,j)\mid \Theta(\alpha^{j+1}(x))\gamma_U(\alpha^j(x))\in (0,1)\}|\\
\leq &\ |\{(U,j)\mid \gamma_U(\alpha^j(x))\in (0,1)\}|\\ 
& \ \ \ +|\{(U,j)\mid \gamma_U(\alpha^j(x))=1 \text{ and } \Theta(\alpha^{j+1}(x))\in (0,1)\}|\\
< &\ \frac{\epsilon_0}{|\U|+1} r_k +|\U|\,|\{0\leq j\leq r_k-2\mid \alpha^{j+1}(x)\in W_0\}|\\
< &\ \frac{\epsilon_0}{|\U|+1} r_k+|\U|\frac{\epsilon_0}{|\U|(|\U|+1)}r_k=\frac{2\epsilon_0}{|\U|+1} r_k \leq \epsilon_0 r_k.
\end{align*} 

Now we estimate how many coordinates $s_l$'s of $\sigma_k(x)$ can be non-zero. We use the fact that if $\gamma_U(\alpha^j(x))=1$, then $\gamma_V(\alpha^j(x))=0$ for any $V$, $V\neq U$.
\begin{align*}
 &\ |\{(U,V,j)\mid U\neq V,\ \gamma_U(\alpha^j(x))\gamma_V(\alpha^j(x))g_{UV}(\alpha^j(x))\neq 0\}|\\   
 = &\ |\{(U,V,j)\mid U\neq V, \ \gamma_U(\alpha^j(x))\gamma_V(\alpha^j(x))\neq 0\}|\\
 \leq &\ |\{(U,V,j)\mid \ \gamma_U(\alpha^j(x))\in (0,1)\ \text{or } \gamma_V(\alpha^j(x))\in (0,1)\}|\\
 \leq &\ 2|\U|\,|\{(U,j)\mid \gamma_U(\alpha^j(x))\in (0,1)\}|\\
 \leq &\ 2|\U|\, \frac{\epsilon_0}{|\U|+1}r_k < 2 \epsilon_0 r_k.
\end{align*}
Therefore $\sigma_k(Y_k)$ is contained in
the set of points 
\[ (t_1, \dots , t_{2|\U|r_k -1}, s_1, \dots , s_{|\U|(|\U|-1) r_k})\in [0,1]^{(2|\mathcal{U}|)r_k-1} \oplus \mathbb{D}^{|\U|(|\U|-1) r_k}\] 
such that $|\{j\mid t_j \notin \{0,1\}\}| < 2\epsilon_0 r_k $ and $|\{l\mid |s_l|\neq 0\}|< 2\epsilon_0 r_k$,  
which shows that $\Delta_k:=\sigma_k(Y_k)$ has dimension 
\[\dim(\Delta_k)\leq 6\epsilon_0 r_k < \epsilon r_k,\] 
since $\dim(\mathbb{D})=2$.

We claim that
\begin{enumerate}[label=(\alph*), resume]
     \item the irreducible representations of $S$ have dimension $r_1 <r_2< \dots < r_K$, and  
     \item $\prim_{r_k}(S) \cong \Delta_k$.
\end{enumerate}

To see (c), it is enough to show that every irreducible representation of $\pi(\mathcal O(\E_Y))$ restricts to an irreducible representation of $\pi(S)$, where $\pi=\oplus_{k=1}^K\pi_k$ is the isomorphism obtained in \cite[Proposition 7.5]{FoJeSt:rsh}.  
By Corollary~\ref{prim A_Y}, since every irreducible representation of $\pi(\mathcal O(\E_Y))$ with dimension $r_k$ is unitarily equivalent to a representation 
\[\nu_{x,k,\mathbf{V},\mathcal A}:\pi(\mathcal O(\E_Y))\to M_{r_k}, \  
\nu_{x,k,\mathbf{V},\mathcal A}(\varsigma_1,\dots,\varsigma_K)=(H_\mathbf{V}^{(k)})^{-1}(\varsigma_k(x))\] for some $x\in Y_k\cap \mathbf{V}$, $\mathbf{V}\in \U^{(r_K)}$, where $\mathcal A=\{(h_U, U)\mid U\in \U\}$ is the atlas fixed in the begining of the proof, we only need to show that 
\[\nu_{x,k,\mathbf{V},\mathcal A}(\pi(S))=M_{r_k}.\] 
From (\ref{Theta xi_U}), we see that the possibly non-zero entries of the first lower subdiagonal matrix   
\begin{align*}
  M_{\pi_k(\Theta\xi_U)}(x) 
=&\ \begin{bmatrix}
0 & 0 & \cdots &  0 & 0 & 0\\
\lambda_0  & 0 & \cdots &  0 & 0 & 0 \\
0 & \lambda_1 & \ddots & 0  & 0 & 0\\
\vdots & \vdots & \ddots & \ddots & \vdots & \vdots \\
0 & 0 & \cdots   & \lambda_{r_k-3} & 0 & 0\\
0 & 0 & \cdots    & 0 & \lambda_{r_k-2} & 0
\end{bmatrix}  
\end{align*} 
are 
\[\lambda_j=\Theta(\alpha^{j+1}(x)) g_{V_j,U}(\alpha^j(x))\gamma_U(\alpha^j(x)),\] 
for $j=0, \dots, r_k-2$. 
Note that $\lambda_j\neq 0$ if and only if $\gamma_U(\alpha^j(x))\neq 0$. Also for each $j$, there is a $U\in \U$ such that $\gamma_U(\alpha^j(x))\neq 0$. 
As we discussed in Remark~\ref{remark:M_{r_k}}, we see that 
the matrices $M_{\pi_k(\Theta\xi_U)}(x)\in \nu_{x,k,\mathbf{U},\mathcal A}(\pi(S))$, $U\in \U$,  generates the whole matrix algebra $M_{r_k}$.  This proves (c).

In order to show (d), we use Lemma~\ref{lemma:SW}.
For $1 \leq k \leq K$, let 
\[J_k := \cap_{n \leq r_k} \cap_{I\in \prim_n(\pi(S))}I.\] 
Then by Corollary~\ref{prim A_Y} and (c), 
\[J_k=\cap\, \{\ker(\nu_{x,i}|_{\pi(S)})\mid x\in Y_i,\ 1\leq i\leq k\}.\] 
(Recall here that $\nu_{x,i}$ denotes any irreducible representation equivalent to $\nu_{x,i,\mathbf{V},\mathcal A}|_{A_Y}$.) 
Set $S_k := J_{k-1}/J_k$ with $J_0:=\pi(S)\subset \pi(\mathcal O(\E_Y))$. 
Then the subquotient $S_k$ is isomorphic to a $\mathrm{C}^*$-subalgebra of $ \Gamma_0(\mathscr{M}_k|_{Y_k})$. 
Actually, the fact that the elements of  
\[J_{k-1}=\{\pi(s)=(0, \dots, 0,\pi_k(s),\dots,\pi_K(s)\mid s\in S\}\subset  \mathrm{ran}(\pi)\] 
must satisfy the boundary decomposition property implies $\pi_k(s)|_{\overline{Y_k}\setminus Y_k}=\{0\}$, that is $\pi_k(s)\in \Gamma_0(\mathscr M_k|_{Y_k})$, and the kernel of the projection homomorphism $\pi(s)\mapsto \pi_k(s):J_{k-1}\to \Gamma_0(\mathscr M_k|_{Y_k})$ is equal to $J_k$, hence  $S_k=J_{k-1}/J_k$ is isomorphic to the following $\mathrm{C}^*$-subalgebra of $\Gamma_0(\mathscr M_k|_{Y_k})$:
\begin{equation}\label{S_k} 
S_k \cong \{\pi_k(s)\in \Gamma_0(\mathscr M_k|_{Y_k})\mid s\in S \text{ with } \pi_1(s)=\cdots =\pi_{k-1}(s)=0\}.
\end{equation}
So, we freely identify $S_k$ and this subalgebra of $\Gamma_0(\mathscr M_k|_{Y_k})$. 
Observe that the following spaces are all homeomorphic:
\begin{equation}\label{prim(S_k)}
 \prim(S_k)\cong \widehat{S}_k\cong \prim_{r_k}(\pi(S))\cong \{\,[\nu_{x,k}|_{\pi(S)}]\mid x\in Y_k\},   
\end{equation} 
where the second homeomorphism is well known (see \cite[Proposition 3.6.3]{Dix:C*}) and the last homeomorphism is obtained by (c) and its proof for each $k=1, \dots, K$.
With $\sigma_k : Y_k \to \Delta_k$, let us verify the conditions of Lemma~\ref{lemma:SW} with $S_k\subset \Gamma_0(\mathscr M_k)$ in place of $S$. This will show that $\prim_{r_k}(S)\cong \prim(S_k)\cong \Delta_k$, which in turn implies (d)  by \eqref{prim(S_k)}.

We first show the following: for $x_1, x_2\in Y_k$, 
\[\sigma_k(x_1)=\sigma_k(x_2) \Leftrightarrow 
[\nu_{x_1, k}|_{\pi(S)}]=[\nu_{x_2, k}|_{\pi(S)}],\] 
which is the condition (1) of Lemma~\ref{lemma:SW}. 
First suppose $x_1, x_2\in Y_k$ satisfy $\sigma_k(x_1)=\sigma_k(x_2)$. Then by definition of $\sigma_k$, we have  for $U,U'\in \U$, 
\begin{align*}
  \gamma_U(\alpha^j(x_1))&=\gamma_U(\alpha^j(x_2)), \, 0\leq j\leq r_k-1,\\
  \Theta(\alpha^{j+1}(x_1))\gamma_U(\alpha^j(x_1))
&=\Theta(\alpha^{j+1}(x_2))\gamma_U(\alpha^j(x_2)), \, 0\leq j\leq r_k-2,\\
g_{UU'}(\alpha^j(x_1))&=g_{UU'}(\alpha^j(x_2)), \, 0\leq j\leq r_k-1.
\end{align*} 
Choose $\mathbf{V}_{i}=(V_{i,0}, \dots, V_{i,r_K-1})\in \U^{(r_K)}$ such that $x_i\in \mathbf{V}_i$ for $i=1,2$. Then  
\[\alpha^j(x_i)\in V_{i,j}\ \text{ for all } i=1,2, \, j=0, \dots, r_K-1,\]  
and the matrix representation  $M_{\pi_k(\Theta\xi_U)}(x_i)$   of $\pi_k(\Theta\xi_U)(x_i)$ (with respect to the basis $\{e_{\mathbf{V}_i,0}(x_i), \dots, e_{\mathbf{V}_i, r_k-1}(x_i)\}$ of $\mathscr D_{x_i}^{(r_k)}$) is a first subdiagonal matrix 
\[M_{\pi_k(\Theta\xi_U)}(x_i)=\diag_{-1}(\lambda_{U,i,0},\dots, \lambda_{U,i,r_k-2}),\]
with entries  
\[\lambda_{U,i,j}=\Theta(\alpha^{j+1}(x_i))\gamma_U(\alpha^j(x_i)) 
 g_{V_{i,j} U}(\alpha^j(x_i))\]
by (\ref{Theta xi_U}) for every $U\in \U$, $i=1,2$, and $0\leq j\leq r_k-2$. 
Thus, it follows from the first condition on $\gamma_U$ that $\lambda_{U,1,j}=0$ if and only if $\lambda_{U,2,j}=0$. 
Then by the third condition on the transition functions $g_{UU'}$   from  $\sigma_k(x_1)=\sigma_k(x_2)$, we have 
\begin{align*}
    g_{V_{2,j}U}(\alpha^j(x_2)) =&\ g_{V_{2,j}U}(\alpha^j(x_1))\\
    =&\ g_{V_{2,j}U}(\alpha^j(x_1))g_{UV_{1,j}}(\alpha^j(x_1))g_{V_{1,j}U}(\alpha^j(x_1))\\
    =&\ g_{V_{2,j}V_{1,j}}(\alpha^j(x_1))g_{V_{1,j}U}(\alpha^j(x_1))\\
    =&\ g_{V_{2,j}V_{1,j}}(\alpha^j(x_2))g_{V_{1,j}U}(\alpha^j(x_1)).
\end{align*}
Note that $\mu_j:=g_{V_{2,j}V_{1,j}}(\alpha^j(x_2))\neq 0$ for every $j=0, \dots, r_k-2$, hence it defines a diagonal $r_k\times r_k$ unitary matrix 
\[\mathfrak u:=\diag(1, \,\mu_0,\, \mu_0\mu_1, \,\dots, \,\mu_{r_k-3}\mu_{r_k-2}),\]
with which we can compute 
\[\mathfrak u M_{\pi_k(\Theta\xi_U)}(x_1) {\mathfrak u}^*= M_{\pi_k(\Theta\xi_U)}(x_2),\] for every $U\in\U$.
Since 
\begin{align*}
 \mathfrak u M_{\pi_k(\gamma_U)}(x_1) \mathfrak u^*= &\ \diag (\gamma_U(x_1), \dots, \gamma_U(\alpha^{r_k-1}(x_1))\\
 =&\ \diag (\gamma_U(x_2), \dots, \gamma_U(\alpha^{r_k-1}(x_2))\\
 =&\ M_{\pi_k(\gamma_U)}(x_2)
\end{align*} 
is obvious,  it follows that $[\nu_{x_1, k}|_{\pi(S)}]=[\nu_{x_2, k}|_{\pi(S)}]$.
For the converse, suppose $[\nu_{x_1, k}|_{\pi(S)}]=[\nu_{x_2, k}|_{\pi(S)}]$. We have to show that $\sigma_k(x_1)=\sigma_k(x_2).$
Let $\mathbf{V}_{i}\in \U^{(r_K)}$ with $x_i\in \mathbf{V}_{i}$ for $i=1,2$.
Then the  matrix representations of the generators of $S_k\subset \Gamma_0(\mathscr M_k|_{Y_k})$ at $x_i$ (with respect to the basis $\{e_{\mathbf{V}_i,0}(x_i), \dots, e_{\mathbf{V}_i, r_k-1}(x_i)\}$) 
are given as follows:
\begin{align*}
   M_{\pi_k(\gamma_U)}(x_i) 
&=\diag\big(
\gamma_U(x_i), \gamma_U(\alpha(x_i)), \dots , \gamma_U(\alpha^{r_k-1}(x_i))
\big),  \\ 
  M_{\pi_k(\Theta\xi_U)}(x_i) 
&= { \begin{bmatrix}
0 & 0 & \cdots & 0 &  0 &\\
 \lambda_{i,0}  & 0 & \cdots &  0 & 0 &\\ 
0 & \lambda_{i,1} & \ddots & \vdots & \vdots &\\
\vdots & \vdots & \ddots & \vdots & \vdots & \\
0 & 0 & \cdots   & \lambda_{i, r_k-2} & 0 &
\end{bmatrix}},
\end{align*} 
where $\lambda_{i,j}=\Theta(\alpha^{j+1}(x_i))g_{V_{i,j}U}(\alpha^j(x))\gamma_U(\alpha^j(x))$.

Since $[\nu_{x_1, k}|_{\pi(S)}]=[\nu_{x_2, k}|_{\pi(S)}]$, there is a unitary matrix \(\mathfrak u\in M_{r_k}\) such that
\begin{align*}
\mathfrak u M_{\pi_k(\gamma_U)}(x_1) \mathfrak u^* &= M_{\pi_k(\gamma_U)}(x_2) ,\\ 
\mathfrak u M_{\pi_k(\Theta\xi_U)}(x_1) \mathfrak u^*  &= M_{\pi_k(\Theta\xi_U)}(x_2)
\end{align*}
for all $U\in \U$. 
Furthermore, $\{ M_{\pi_k(\Theta\xi_U)}(x_i)\}_{U\in \U}$, $i=1,2$,  generates  the whole matrix algebra $M_{r_k}$  as a $\mathrm{C}^*$-algebra (as discussed in the proof of (c)), this implies that $Ad_{\mathfrak u}:M_{r_k}\to M_{r_k}$ is an automorphism sending  $j$-th  subdiagonal matrices to  $j$-th subdiagonal matrices for each  $j=0,\dots, r_k-1$.
Then a straightforward inductive argument implies that such a unitary $\mathfrak u$ must be a diagonal matrix. So 
\[M_{\pi_k(\gamma_U)}(x_2)  = \mathfrak u M_{\pi_k(\gamma_U)}(x_1) \mathfrak u^* = M_{\pi_k(\gamma_U)}(x_1)\] from which we have
\[
\gamma_U(\alpha^j(x_1))=\gamma_U(\alpha^j(x_2)), \ j=0,\ldots, r_k-1.
\]

Thus  $\gamma_U(\alpha^j(x_1))\neq 0$ if and only if $\gamma_U(\alpha^j(x_2))\neq 0$ for every $U\in \U$, and so we can always take the same  open neighborhood $V_j:=V_{1j}=V_{2j}\in \U$ of $\alpha^j(x_1)$ and $\alpha^j(x_2)$, respectively for $j=0,\dots,r_k-1$. 
Then for the choice of basis
\[
 h_{V_j}(\alpha^{j}(x_i),1)\in \mathscr V_{\alpha^{j}(x_i)}, \ \ 0\leq j\leq r_k-1, \ i=1,2,
\]
the matrix representations 
$M_{\pi_k(\Theta\xi_U)}(x_i)$ for $U\in \U$, $i=1,2$, 
are 
\begin{align*}
&\ M_{\pi_k(\Theta\xi_U)}(x_i)\\
=\ &  \begin{bmatrix}
0 & 0 & \cdots &  \\
\Theta(\alpha(x_i))g_{V_0 U}(x_i)\gamma_U(x_i)  & 0 & \cdots &   \\
0 & \Theta(\alpha^2(x_i))g_{V_1 U}(\alpha(x_i))\gamma_U(\alpha(x_i))& \ddots & \\
\vdots & \vdots & \ddots &  \\
0 & 0 & \cdots   & 
\end{bmatrix}\end{align*} as seen before. 
Since the diagonal unitary matrix $\mathfrak u$ satisfies 
\[\mathfrak u M_{\pi_k(\Theta\xi_U)}(x_1) \mathfrak u^* = M_{\pi_k(\Theta\xi_U)}(x_2),\] 
there are scalars $\mu_j\in \mathbb T$ (equal to some product  of entries of $\mathfrak u$), $j=0,\ldots, r_k-1$,
such that 
\begin{align*}
&\ \begin{bmatrix}
0 & 0 & \cdots &  \\
\mu_0\Theta(\alpha(x_1))g_{V_0 U}(x_1)\gamma_U(x_1)  & 0 & \cdots &   \\
0 & \mu_1\Theta(\alpha^2(x_1))g_{V_1 U}(\alpha(x_1))\gamma_U(\alpha(x_1))& \ddots & \\
\vdots & \vdots & \ddots &  \\
0 & 0 & \cdots   & 
\end{bmatrix}\\
=&\  \begin{bmatrix}
0 & 0 & \cdots &  \\
\Theta(\alpha(x_2))g_{V_0 U}(x_2)\gamma_U(x_2)  & 0 & \cdots &   \\
0 & \Theta(\alpha^2(x_2))g_{V_1 U}(\alpha(x_2))\gamma_U(\alpha(x_2))& \ddots & \\
\vdots & \vdots & \ddots &  \\
0 & 0 & \cdots   & 
\end{bmatrix}.
\end{align*}
Thus we have 
\begin{equation}\label{mu_j}
\mu_j \,\Theta(\alpha^{j+1}(x_1))g_{V_j  U}(\alpha^j(x_1))  = \Theta(\alpha^{j+1}(x_2)) g_{V_j U}(\alpha^j(x_2))
\end{equation} 
for any $U\in \U$.
Choosing $U=V_j$ (so $g_{V_j, V_j}(\alpha^j(x_i))=1$), we have 
\[\mu_j=\Theta(\alpha^{j+1}(x_2))/\Theta(\alpha^{j+1}(x_1)), \] 
which proves that the number $\mu_j\in \mathbb T$ has to be one and thus  we have 
\[  \Theta(\alpha^{j+1}(x_1))=\Theta(\alpha^{j+1}(x_2)), \ j=0, \dots, r_k-2.\]
This also implies, together with, \eqref{mu_j}  that 
\[
g_{V_j  U}(\alpha^j(x_1))=g_{V_j U}(\alpha^j(x_2)).
\]
Since $V_j\in \U$ could be any neighborhood of $\alpha^j(x_1)$ (and hence $\alpha^j(x_2)$),  we have
\begin{align*}
&\ \gamma_U(\alpha^j(x_1))\gamma_V(\alpha^j(x_1))g_{U,U'}(\alpha^j(x_1))\\
=&\ \gamma_U(\alpha^j(x_2))\gamma_V(\alpha^j(x_2))g_{U,U'}(\alpha^j(x_2)),\ U,U'  \in \U.    
\end{align*}
Therefore  $[\nu_{x_1,k}|_{\pi(S)}]=[\nu_{x_2,k}|_{\pi(S)}]$ implies that $\sigma_k(x_1)=\sigma_k(x_2)$ for every $k=1,\dots, K$, which completes the proof for condition (1) of Lemma~\ref{lemma:SW}.

To prove the condition (2) of Lemma~\ref{lemma:SW}, 
recall that using \eqref{S_k} we identify $S_k$ as a subalgebra of $\Gamma_0(\mathscr M_k|_{Y_k})$, and by \eqref{prim(S_k)} every irreducible representation of $S_k$  can be viewed as an irreducible representation of $\pi(S)$  equivalent to a $r_k$-dimensional one $\nu_{x,k,\mathbf{V},\mathcal A}|_{\pi(S)}$ for some $x\in Y_k$ and $\mathbf{V}\in \U^{(r_K)}$ with $x\in \mathbf{V}$.  
 Without explicitly mentioning $\mathbf{V}$ and $\mathcal A$, and by abuse of notation, we write such an irreducible representation of $S_k$ as
\[\nu_{x,k}|_{S_k}(\pi_k(s)):= \nu_{x,k,\mathbf{V},\mathcal A}(\pi(s))\in M_{r_k}, \] 
for $\pi_k(s)\in S_k$ ($s\in S$ is an element such that $\pi(s)\in J_{k-1}$).

 Let $x_i, x\in Y_k$ and $\sigma_k(x_i)\to \sigma_k(x)$ as $i\to \infty$.  Suppose that $\pi_k(a) \in S_k\subset \Gamma_0(\mathscr M_k|_{Y_k})$   satisfies  
\begin{equation}\label{0 at x_i}
 \pi_k(a)(x_i)=0 \ \text{ for all } i\geq 1.   
\end{equation} 
We have to show that $\pi_k(a)(x)=0$. 
The condition \eqref{0 at x_i})is nothing but 
\[\nu_{x_i,k}|_{S_k}(\pi_k(a))=0\ \text{ for all } i\geq 1,\]
and to obtain $\pi_k(a)(x)=\nu_{x,k}|_{S_k}(\pi_k(a))=0$, we show that 
\[\|\nu_{x,k}|_{S_k}(\pi_k(a))\|<\delta\ \text{ for any }\ \delta>0.\] 
Choose $\mathbf{V}=(V_0, \dots, V_{r_K-1})\in \U^{(r_K)}$ with $x\in \mathbf{V}$ and $\gamma_{V_j}(\alpha^j(x))\neq 0$ for all $j=0, \dots, r_K-1$. 
Since $\sigma_k(x_i)\to \sigma_k(x)$ as $i\to \infty$, there is an $N\in \mathbb N$ such that for each $j=0, 1,\dots , r_K-1$,
\[|\gamma_{V_j}(\alpha^j(x_i))-\gamma_{V_j}(\alpha^j(x))|
< \min_{0\leq l\leq r_K-1}\gamma_{V_l}(\alpha^l(x)), \ \text{for all } i\geq N.
\] 
Then $\alpha^j(x_i)\in V_j$ for all $0\leq j\leq r_k-1$ and $i\geq N$.
For $x'=x_i$, $i\geq N$, or $x'=x$, consider the matrix representations 
\[\nu_{x',k}|_{S_k}(\pi_k(s))=M_{\pi_k(s)}(x')\] of $\pi_k(s)(x')$ for the generators $s\in \{\gamma_U, \Theta\xi_U\mid U\in \U\}$ of $S$ with respect to the basis $\{e_{\mathbf{V},j}(x_i)\mid 0\leq j\leq r_k-1\}$ of $\mathscr D_{x_i}^{(r_k)}$ for $i\geq N$.  
Then  $\sigma_k(x_i)\to \sigma_k(x)$ shows that 
\[M_{\pi_k(\gamma_U)}(x_i)=\diag(\gamma_U(x_i), \gamma_U(\alpha(x_i)), \dots, \gamma_U(\alpha^{r_k-1}(x_i)))\]
converges to 
\[M_{\pi_k(\gamma_U)}(x)=\diag(\gamma_U(x), \gamma_U(\alpha(x)), \dots, \gamma_U(\alpha^{r_k-1}(x))),\] 
and 
\[M_{\pi_k(\Theta\gamma_U)}(x_i)=\diag_{-1}(\lambda_{i0}, \lambda_{i_1}, \dots, \lambda_{i r_k-1}),\] where $\lambda_{ij}=\Theta(\alpha^{j+1}(x_i)\gamma_U(\alpha^j(x_i)g_{V_jU}(\alpha^j(x_i))$  converges to 
\[M_{\pi_k(\Theta\gamma_U)}(x)=\diag_{-1}(\lambda_{0}, \lambda_{1}, \dots, \lambda_{r_k-1}), \]
as $i\to \infty$, for every $U\in \U$, where $\lambda_{j}=\Theta(\alpha^{j+1}(x)\gamma_U(\alpha^j(x)g_{V_jU}(\alpha^j(x))$. 

Now, choose a polynomial 
\[p(s_1,\dots,s_{|\U|}, \hat{s}_1,\dots, \hat{s}_{|\U|}, t_1,\dots,t_{|\U|}, \hat{t}_1,\dots, \hat{t}_{|\U|})\] with $4|\U|$ variables such that 
\[\|a-p(\gamma_{U_l}, \Theta\xi_{U_l})\|<\delta,\] 
where by $p(\gamma_{U_l}, \Theta\xi_{U_l})$ we briefly denote the polynomial of  $\gamma_{U_l}$ and $\Theta\xi_{U_l}$ in place of $s_l$ and $t_l$, respectively, and their involutions in place of $\hat{s}_l$ and $\hat{t}_l$, respectively. 
Then 
\[ \|\nu_{x,k}|_{S_k}(\pi_k(a))-p(M_{\pi_k(\gamma_{U_l})}(x), M_{\pi_k(\Theta\xi_{U_l})}(x))\|<\delta\] 
where 
$p(M_{\pi_k(\gamma_{U_l})}(x), M_{\pi_k(\Theta\xi_{U_l})}(x))
\in M_{r_k}$  is the matrix when  
\[s_l=M_{\pi_k(\gamma_{U_l})}(x),\ t_l= M_{\pi_k(\Theta\xi_{U_l})}(x)\] 
and $\hat{s}_l=M_{\pi_k(\gamma_{U_l})}(x)^*$, $\hat{t}_l=M_{\pi_k(\Theta\xi_{U_l})}(x)^*$ for $l=1,\dots,|\U|$. 
Since 
\[M_{\pi_k(\gamma_{U_l})}(x_i)\to M_{\pi_k(\gamma_{U_l})}(x) \text{ and }
M_{\pi_k(\Theta\xi_{U_l})}(x_i)\to M_{\pi_k(\Theta\xi_{U_l})}(x)\] as $i\to \infty$, we have 
\begin{align*}&\ \|\nu_{x,k}|_{S_k}(\pi_k(a))-\nu_{x_i,k}|_{S_k}(\pi_k(a))\|\\
\leq &\ \|\nu_{x,k}|_{S_k}(\pi_k(a))-p(M_{\pi_k(\gamma_{U_l})}(x), M_{\pi_k(\Theta\xi_{U_l})}(x))\|\\ 
&\ +\|p(M_{\pi_k(\gamma_{U_l})}(x), M_{\pi_k(\Theta\xi_{U_l})}(x))-p(M_{\pi_k(\gamma_{U_l})}(x_i), M_{\pi_k(\Theta\xi_{U_l})}(x_i))\|\\
&\ +\|p(M_{\pi_k(\gamma_{U_l})}(x_i), M_{\pi_k(\Theta\xi_{U_l})}(x_i))-
\nu_{x_i,k}|_{S_k}(\pi_k(a))\|\\
<&\ 3\delta
\end{align*} for all $i$ sufficiently large. Thus $\|\nu_{x,k}|_{S_k}(\pi_k(a))\|<3\delta$ follows whenever $\nu_{x_i,k}|_{S_k}(\pi_k(a))=0$ for all $i$.

Since $\nu_{x,k}|_{S_k}(\pi_k(S))=\nu_{x,k}(\pi(S))=M_{r_k}$ is already seen as mentioned in (\ref{prim(S_k)}) and the proof of (c), we have $\prim(S_k)=\Delta_k$ by Lemma~\ref{lemma:SW}. 
Thus $\dim \mathrm{Ratio}(S) < \epsilon$ follows from  $\dim(\Delta_k) < \epsilon r_k$. 

It remains to show that for $f_i\in \mathcal F_1$, $g_i\in \mathcal F_2$, and  $\xi_i\in \E$, $\|\xi_i\|\leq 1$, 
\[\dist(f_i, S)<\epsilon\ \text{ and }\  \dist(g_i\xi_i,S)<\epsilon,\] 
$i=1,\dots, m$. This follows from our choice of cover $\mathcal{U}$. Indeed, let $x_U \in U$. 
Then by (a) 
 \[\| f_i - \sum_{U \in \mathcal{U}} f_i(x_U)\gamma_U \|  \leq \sup_{x \in X} \sum_{U \in \mathcal{U}} |f_i(x) - f_i(x_U)| \gamma_U(x) < \epsilon_0 < \epsilon.\]
Recall that $\{\eta_U\}_U\subset\E$ defined by $\eta_U(x)=h_U(x,\gamma_U(x)^{1/2})$ generates $\E$ as a right Hilbert $C(X)$-module, that is $\xi=\sum_U\eta_U\langle \eta_U,\xi\rangle_\E$ for all $\xi\in \E$. 
Let $g:=g_i$ ($\Theta g=g$ since $g\in C_0(X\setminus W)$) and $\xi:=\xi_i$ for some $i$, and let $s_U$ be the continuous function such that   $\xi(x)=h_U(x,s_U(x))$ for $x\in U$.
Since $\langle \eta_U,\xi\rangle_\E(x)=\langle \gamma_U(x)^{1/2}, s_U(x)\rangle_{\mathbb C}=\gamma_U(x)^{1/2}s_U(x)$, we write 
\[(\eta_U\langle\eta_U,\xi\rangle_\E)(x)=h_U(x,\gamma_U(x)^{1/2})\gamma_U(x)^{1/2}s_U(x)=\xi_U(x)s_U(x).\]
Then for $\sum_U g(\alpha(x_U))s_U(x_U)\Theta\xi_U\in S$,  we  have
\begin{align*}
&\ \|g\xi-    \sum_U g(\alpha(x_U))s_U(x_U)\Theta\xi_U\|\\
=&\ \|\Theta g\sum_U\eta_U\langle \eta_U,\xi\rangle_\E-    \sum_U g(\alpha(x_U))s_U(x_U)\Theta\xi_U\|\\
=&\ \|\sum_U \Theta g \xi_Us_U-    \sum_U g(\alpha(x_U))s_U(x_U)\Theta\xi_U\|\\
=&\ \|\sum_U \Theta g \xi_U(\cdot,1)s_U\gamma_U-    \sum_U g(\alpha(x_U))s_U(x_U)\Theta\xi_U(\cdot,1)\gamma_U\|\\
\leq &\ \|\sum_U \Big(g \xi_U(\cdot,1)s_U-    g(\alpha(x_U))s_U(x_U)\xi_U(\cdot,1)\Big)\gamma_U\|\\
< &\ \epsilon_0 < \epsilon,
\end{align*}
where the last inequality follows from (b) and 
\begin{align*} 
&\ |\, h_U^{-1}\big((g \xi_U(\cdot,1)s_U-    g(\alpha(x_U))s_U(x_U)\xi_U(\cdot,1))(x)\big)\,|\\ 
=&\ |\, g(\alpha(x))s_U(x)-g(\alpha(x_U))s_U(x_U)\,|< \epsilon,
\end{align*}  for $x\in U$.
\end{proof}

This immediately gives  us the following theorem.

\begin{theorem} \label{thm:locally sub hom}
Let $X$ be an infinite metrizable compact space,  $\alpha : X \to X$ be a minimal homeomorphism,  $\mathscr{V}$ a line bundle, and let $y \in X$. If $(X, \alpha)$ has the small boundary property, then the $\mathrm C^*$-algebra $\mathcal{O}(\E_{\{y\}})$ can be locally approximated by subhomogeneous $\mathrm C^*$-subalgebras with arbitrarily small dimension ratio.
\end{theorem}

\begin{theorem} \label{thm:Z-stable}
Let $X$ be an infinite compact metric space, $\alpha : X \to X$ a minimal homeomorphism and $\mathscr V$ a line bundle over $X$. Suppose $(X, \alpha)$ has the small boundary property. Then $\mathcal O_{C(X)}(\Gamma(\mathscr{V}, \alpha))$ is $\mathcal Z$-stable.
\end{theorem}

\begin{proof}
Let $X$ be a compact metric space and $(X, \alpha)$ be a minimal dynamical system with the small boundary property and let $p : \mathscr{V} \to X$ be a line bundles. Since $(\alpha, X)$ is minimal, $\mathcal{O}(\E)$ is simple \cite[Corollary 3.11]{AAFGJSV2024}. Suppose $Y \subset X$ is a closed subset meeting every $\alpha$-orbit at most once such that for every $N \in \mathbb{Z}_{>0}$ there exists an open set $W_N \supset Y$ such that $\mathscr{V}|_{\alpha^n(W_N)}$ is trivial for every $-N \leq n \leq N$.  Then $\mathcal{O}(\E_Y)$ is (centrally) large subalgebra of $\mathcal{O}(\E)$ \cite[Theorem 6.14]{AAFGJSV2024}, hence is also simple.

Let $y \in Y$. By Theorem~\ref{thm:locally sub hom}, $\mathcal{O}(\E_{\{y\}})$ can be locally approximated by subhomogeneous $\mathrm C^*$-algebras of arbitrarily small dimension ratio. Since the approximating recursive subhomogeneous algebras in Theorem~\ref{thm:FindingS} have topological dimension not exceeding $r_K \epsilon$, they have finite nuclear dimension \cite[Theorem 6.1]{Winter:subhomdr}. Hence $\mathcal{O}(\E_{\{y\}})$ has locally finite nuclear dimension. Using \cite[Theorem 5.1]{Toms:CompSmoothDyn}, we see that the radius of comparison of a recursive subhomogenous $\mathrm C^*$-algebra is bounded above by its dimension ratio, so, applying \cite[Lemma 5.8]{Niu:AHMeanDim}, $\mathcal{O}(\E_{\{y\}})$ has strict comparison of positive elements. If $\partial_e T(A_{\{y\}})$ is compact, this implies $\mathcal{Z}$-stability \cite{KirRor:CentralSeqZ-Stab, Sato:FinDimTrace, TomWhiWin:findim_Trace}. Otherwise, as in \cite[Theorem 4.7]{EllNiu:MeanDimZero}, we may appeal to the argument in the proof of \cite[Theorem 1.2]{toms:ksd} using \cite[Lemma 5.10]{Niu:AHMeanDim} in place of \cite[Theorem 3.4]{toms:ksd}, to show that the Cuntz semigroups of $\mathcal{O}(\E_{\{y\}})$ and $\mathcal{O}(\E_{\{y\}}) \otimes \mathcal{Z}$ are isomorphic, and therefore $\mathcal{O}(\E_{\{y\}}) \cong \mathcal{O}(\E_{\{y\}}) \otimes \mathcal{Z}$ by \cite[Corollary 7.4]{Win:pure}.

By \cite[Theorem 6.14]{AAFGJSV2024}, $\mathcal{O}(\E_{\{y\}})$ is a centrally large subalgebra. Thus $\mathcal{O}(\E)$ is also $\mathcal{Z}$-stable \cite[Theorem 3.3]{ArBkPh-Z}.
\end{proof}

Since mean dimension zero is equivalent to having the small boundary property whenever $(X, \alpha)$ is minimal, we obtain the following corollary. 

\begin{corollary}
        Let $X$ be an infinite compact metric space, $\alpha : X \to X$ a minimal homeomorphism and $\mathscr V$ a line bundle over $X$. Suppose that $(X, \alpha)$ has mean dimension zero. Then $\mathcal O_{C(X)}(\Gamma(\mathscr{V}, \alpha))$ is $\mathcal Z$-stable.
    \end{corollary}

We also obtain a sufficient condition for $\mathcal Z$-stability in terms of the tracial state space of $\mathcal O_{C(X)}(\Gamma(\mathscr{V}, \alpha))$.

\begin{corollary}
        Let $X$ be an infinite compact metric space, $\alpha : X \to X$ a minimal homeomorphism and $\mathscr V$ a line bundle over $X$. Let $A  = \mathcal O_{C(X)}(\Gamma(\mathscr{V}, \alpha))$. Suppose $T(A)$ is Bauer simplex. Then $\mathcal O_{C(X)}(\Gamma(\mathscr{V}, \alpha))$ is $\mathcal Z$-stable.
    \end{corollary}
    \begin{proof}
        By \cite[Proposition 4.5]{AAFGJSV2024}, there is an affine homeomorphism $T(A) \cong M^1(X, \alpha)$, where $M^1(X, \alpha)$ is the space of $\alpha$-invariant Borel probability measures. Thus 
        $(X, \alpha)$ has the small boundary property \cite[Theorem 4.6]{EllNiu:SBP}, from which it follows that $\mathcal O_{C(X)}(\Gamma(\mathscr{V}, \alpha))$ is $\mathcal Z$-stable.  
    \end{proof}

 \begin{example} \label{ex:MDhom}
        Let $Q$ denote the Hilbert cube. Then $X = Q \times [0,1]$ is a Hilbert cube manifold. As pointed out on page 323 of \cite{GlasWeiss79}, there is a path-connected subgroup $\Gamma \subset \mathrm{Homeo}(X)$ such that $(X, \Gamma)$ is minimal, (see also  page~442 of~\cite{DirbMarMal2011}). Identify $\mathbb T^2$ with $\mathbb R^2 / \mathbb Z^2$, and for $z \in \mathbb R$, let $[z]$ denote its image in $\mathbb R /\mathbb Z$. Define 
        \[\alpha : \mathbb T^2 \to \mathbb T^2
        \]
        by 
        \[ \alpha(x, y) = ([x + \mu], [y + \nu]),\]
        for irrational numbers $\mu, \nu$. When $1, \mu, \nu$ are rationally independent, $(\mathbb T^2, \alpha)$ is minimal and uniquely ergodic. 
        
        Applying the main results of \cite{GlasWeiss79}, there exists a minimal homeomorphism $\beta : X \times \mathbb T^2 \to X \times \mathbb T^2$ which is uniquely ergodic. In particular the space $X \times \mathbb T^2$ is infinite-dimensional but $(X \times \mathbb T^2 , \beta)$ has mean dimension zero. 
        
        Since $Q$ and $[0,1]$ are both contractible, $X \times \mathbb T^2$ admits non-trivial complex line bundles by pulling back any line bundle over $\mathbb T^2$. Let $\mathscr V$ be such a bundle. Then $\mathcal O_{C(X \times \mathbb T^2)}(\Gamma(\mathscr{V}, \beta))$ is simple, separable, unital, nuclear and $\mathcal Z$-stable.
    \end{example}

\begin{notation} \label{note:C}
Let $\mathcal{C}$ denote the class of $\mathrm{C}^*$-algebras of the form  $\mathcal{O}(C_0(X \setminus Y)\Gamma(\mathscr{V}, \alpha))$ where $X$ infinite compact metric space, $\mathscr{V} = [V, p, X]$ is a line bundle, $\alpha : X \to X$ a minimal homeomorphism, and $Y \subset X$ is a closed subset meeting every $\alpha$-orbit at most once such that for every $N \in \mathbb{Z}_{>0}$ there exists an open set $W_N \supset Y$ such that $\mathscr{V}|_{\alpha^n(W_N)}$ is trivial for every $-N \leq n \leq N$. Note that if $Y = \emptyset$ we have $\mathcal{O}(\E_Y) = \mathcal{O}(\E)$.    

Let $\mathcal{C}_0 \subset \mathcal{C}$ denote the subclass consisting of the $\mathrm C^*$-algebras in $\mathcal{C}$ where the underlying dynamical system has the small boundary property (or, equivalently, mean dimension zero).
\end{notation}

\begin{theorem}[see for example \cite{CETWW, BBSTWW:2Col, EllGonLinNiu:ClaFinDecRan, GongLinNiue:ZClass, GongLinNiue:ZClass2, TWW, CGSTW:ClassI}]
 \label{ClassThm} Let $A$ and $B$ be separable, unital, simple, $\mathcal{Z}$-stable \mbox{$\mathrm{C}^*$-algebras} and which satisfy the UCT. Suppose there is an isomorphism 
\[ \psi : \Ell(A) \to \Ell(B).\]
Then there is a $^*$-isomorphism 
\[ \Psi : A \to B,\]
satisfying $\Ell(\Psi) = \psi$.
\end{theorem}

\begin{theorem} \label{thm:classification1}
 Suppose that $A, B \in \mathcal{C}_0$ and 
	\[ \psi : \Ell(A) \to \Ell(B) \]
	is an isomorphism. Then there exists a $^*$-isomorphism 
	\[ \Psi : A \to B,\]
satisfying $\Ell(\Psi) = \psi$.
\end{theorem}

\begin{proof}
It is enough to show that if $A \in \mathcal{C}_0$ then $A$ is separable, unital, simple, $\mathcal{Z}$-stable and satisfies the UCT. Any $A \in \mathcal{C}_0$ is separable and unital. Moreover, \cite[Proposition 8.8]{Katsura2004}, since $C(X)$ is commutative, any $A \in \mathcal{C}$ satisfies the UCT. 

 If $Y \subset X$ is a closed subset meeting every $\alpha$-orbit at most once such that for every $N \in \mathbb{Z}_{>0}$ there exists an open set $W_N \supset Y$ such that $\mathscr{V}|_{\alpha^n(W_N)}$ is trivial for every $-N \leq n \leq N$, then $\mathcal{O}(\E_Y)$ a centrally large subalgebra of $\mathcal{O}(\E)$ \ref{centrally large}. Thus $\mathcal Z$-stability passes from $\mathcal O(\E)$ to $\mathcal O(\E_Y)$. Since $(X, \alpha)$ is a minimal dynamical with the small boundary property, $\mathcal O(\E)$ is $\mathcal Z$-stable by Theorem~\ref{thm:Z-stable}.

Thus every $A \in \mathcal{C}_0$ is separable, unital, simple, $\mathcal{Z}$-stable and satisfies the UCT. This proves the theorem.
\end{proof}


\section{Tensor products} \label{sec:TP}

Giol and Kerr have constructed examples of minimal dynamical systems giving rise to crossed products which are not $\mathcal{Z}$-stable. Since such crossed products are isomorphic to Cuntz--Pimsner algebra $\mathcal{O}(\E)$ with $\E = \Gamma(\mathscr{V}, \alpha)$ for $\mathscr{V}$ a trivial vector bundle, it is unlikely that we will be able to enlarge the class $\mathcal{C}_0$ of Notation~\ref{note:C}. 

Nevertheless, we are able to say something about tensor products of $\mathrm C^*$-algebras in the larger class $\mathcal{C}$ (Notation~\ref{note:C}). Indeed, we are able to adapt results of \cite{EllNiu:MeanDimZero} to our setting to show that, regardless of mean dimension, tensor products of the form $\mathcal{O}(\Gamma(\mathscr{V}, \alpha)) \otimes \mathcal{O}(\Gamma(\mathscr{W}, \beta))$, for $(\alpha,X), (\beta,Y)$ minimal and $\mathscr{V}, \mathscr{W}$ line bundles over $X, Y$ respectively, are always $\mathcal{Z}$-stable.

\begin{theorem} \label{thm:SubinTensorProd}
Let $X$ be an infinite compact metric space, $\mathscr{V}$ a line bundle over $X$ and $\alpha : X \to X$ a minimal homeomorphism. Let $\E = \Gamma(\mathscr{V}, \alpha)$. Fix $y \in X$ and  $W \subset X$ an open neighborhood of $y$. Suppose $\F_1 = \{ f_1, \dots, f_m\} \subset C(X)$, $\F_2 = \{g_{1}, \dots, g_{m} \} \subset C_0(X \setminus W)$,  $\{\xi_1,\dots,\xi_m\mid \xi_i\in \E,\ \|\xi_i\|\leq 1\}$ are given. Then for any $\epsilon > 0$ there is  $R>0$ such that for any $J \in \mathbb N$, there exists $Y \subset W$  a closed neighborhood of $y$ and a subhomogeneous $C^*$-algebra $S \subset \mathcal{O}(\E_Y)$ such that 
\begin{enumerate}[label=\textup{(\arabic*)}]
\item $\dim\textup{Ratio} (S) \leq R$, 
\item the dimension of each irreducible representation is at least $J$, 
\item $\dist(f_i, S) < \epsilon$, for every $1 \leq i \leq m$, 
\item $\dist(g_{i} \xi_i, S) < \epsilon$, for every $1 \leq i \leq m$.
\end{enumerate}
\end{theorem}

\begin{proof}
As in Theorem~\ref{thm:FindingS}, if  $(U,h_U)$ is a chart of $\mathscr V$,  we denote by $s_{i,U}\in C(X)$ continuous function induced by the section $\xi_i|_U$. In particular, $s_{i,U}$ satisfies
\[\xi_i(x)=h_U (x, s_{i,U}(x)) \ \text{for }\, x \in U.\]
Fix a finite open cover $\mathcal{U}$ of $X$ such that $\mathscr V|_U$, $U\in \U$, is trivial with chart map $h_U:U\times \mathbb C\to \mathscr V|_U$  and the following hold:
\begin{enumerate}[label=(\alph*)]
  \item   $|f_i(x) - f_i(x') | < \epsilon$, 
  \item \( |g_i(\alpha(x))s_{i,U}(x)-g_i(\alpha(x'))s_{i,U}(x')| < \epsilon \),
\end{enumerate}
for every $x, x' \in U$, $U \in \mathcal{U}$, and  $1 \leq i \leq m$. For $U,U'\in \U$, $g_{U,U'}:U\cap U'\to \mathbb U(1)$ denotes the transition function satisfying
\[h_U(x,\lambda)=h_V(x, g_{U',U}(x)\lambda),\ \lambda\in \mathbb C,\  x\in U\cap U'.\]
We put $ g_{U,U'}=0$ if $U \cap U' =\emptyset$ for $U,U' \in \U$.

Let $R := 2|\mathcal U|^2$ and let $J \in \mathbb Z_{>0}$ be arbitrary. 

As in the proof of Theorem~\ref{thm:FindingS}, by minimality of $\alpha$ there is a closed neighbourhood $Y \subset W$ of $y$ with first return times $r_1 < r_2 < \dots < r_K$ satisfying $J \leq r_1$.

Let $W_0 \subset W$ be an open set containing $Y$, and let $\Theta : X \to [0,1]$ be a function satisfying $\Theta(x) = 0$ if and only if $x \in Y$ and $\Theta|_{X \setminus  W_0} = 1$. Note from $W_0 \subset W$  that $\Theta g_i = g_i$ for every $1 \leq i \leq m$. 

Now consider the $C^*$-subalgebra 
\[S:=C^*(\{\gamma_U\}_{U\in \mathcal U},\, \{\Theta\xi_U\}_{U\in \mathcal U})\subset \mathcal O(\E_Y).\]
Note that $S$ is subhomogeneous. We show that $S$ and $Y$ satisfy the conditions of the Theorem for $R$ and $J$.

Define $\gamma_\U$, $\Psi :X\to \mathbb R^{|\U|}$ by 
\[\gamma_\U(x):=\bigoplus_{U\in \U} \gamma_U(x) \ \ \text{ and } \ \Psi(x)=\bigoplus_{U \in \mathcal{U}} \Theta(\alpha(x))\gamma_U(x).\]  
Let $\mathbb D:=\{z\in \mathbb C\mid |z|\leq 1\}$ and define   $\Phi : X \to \mathbb{D}^{|\U|(|\U|-1)}$ by
\[  \Phi(x) := \bigoplus_{U\neq V \in \mathcal{U}} \gamma_U(x)\gamma_V(x) g_{UV}(x).\] 
For every $1 \leq k \leq K$, let
 \[ \sigma_k : \overline{Y_k} \to \mathbb{R}^{|\mathcal{U}|(2r_k-1)} \oplus \mathbb{D}^{|\U|(|\U|-1) r_k}\] 
 be the map given by
 \begin{align*} 
 \sigma_{k}(x) =& (\gamma_\U(x), \gamma_\U \circ\alpha(x), \dots, \gamma_\U\circ \alpha^{r_k-1}(x), \\ 
  &\ \ \ \ \Psi(x), \Psi \circ\alpha(x), \dots, \Psi\circ \alpha^{r_k-2}(x), \\
 & \ \ \ \ \ \ \ \Phi(x), \Phi\circ\alpha(x), \dots, \Phi\circ\alpha^{r_k-1}(x)). 
 \end{align*} 
 
Let $\Delta_k = \sigma_k(Y_k)$. Then 
\begin{align*}
\dim (\Delta_k) &\leq \dim( \mathbb{R}^{|\mathcal{U}|(2r_k-1)} \oplus \mathbb{D}^{|\U|(|\U|-1) r_k})\\
&= 2 |\mathcal U|^2 r_k - |\mathcal U| \\
&< 2 |\mathcal U|^2 r_k .
\end{align*}

As in the proof of Theorem~\ref{thm:FindingS}, we have $\prim_{r_k}(S) \cong \Delta_k$, and the irreducible representations of $S$ have dimension $r_1< r_2 < \cdots < r_K$. It follows that 
\[ \dim \mathrm{Ratio}(S) < 2|\mathcal U |^2 = R,\]
showing (1) and (2). The proof that $\dist(f_i, S) < \epsilon$ and  $\dist(g_i \xi_i, S) < \epsilon$ for every $1 \leq i \leq m$ is the same as in Theorem~\ref{thm:FindingS}, showing (3) and (4).
\end{proof}
 
The \emph{growth rank} of a $\mathrm C^*$-algebra $A$, introduced by Toms in \cite{Toms:growth}, is the least non-negative integer $n$ such that
\[ A^{\otimes n} = \underbrace{A \otimes_{\min} A \otimes_{\min} \cdots \otimes_{\min} A}_{\text{$n$ times}}\]
is $\mathcal{Z}$-stable, assuming the minimal tensor product. If no such integer exists, then we say $A$ has infinite growth rank \cite[Definition 2.1]{Toms:growth}.

\begin{theorem}
Let $A, B \in \mathcal{C}$. Then
 \[ A \otimes B \cong A \otimes B \otimes \mathcal{Z}.\]
In particular, if $A \in \mathcal{C}$ then $A$ has growth rank $2$.
\end{theorem}

\begin{proof}
 Let $X$ and $Y$ be infinite compact metrizable spaces, $(\alpha,X), (\beta,Y)$  minimal dynamical systems, and $\mathscr{V}, \mathscr{W}$ line bundles over $X, Y$, respectively.  Let $x \in X$ and $y \in Y$. Set $\E = \Gamma(\mathscr{V}, \alpha)$ and $\F = \Gamma(\mathscr{W}, \beta)$, and let $\E_{\{x\}} = C_0(X \setminus \{x\}) \E$ and $\F_{\{y\}} = C_0(Y \setminus \{y\}) \F$. 

 By Theorem~\ref{thm:SubinTensorProd}, for any finite subset $\F_1 \subset \mathcal{O}(\E_{\{x\}})$, every finite subset $\F_2 \subset \mathcal{O}(\F_{\{y\}})$ and any $\epsilon > 0$ there exists $R >0$ and sequences of unital $\mathrm{C}^*$-subalgebras $(S_n)_{n \in \mathbb{N}} \subset \mathcal{O}(\E_{\{x\}})$,  $(T_n)_{n \in \mathbb{N}} \subset \mathcal{O}(\F_{\{x\}})$ such that
 \begin{enumerate}[label=(\alph*)]
     \item for every $n \in \mathbb{N}$, $S_n$ and $T_n$ are subhomogeneous $\mathrm C^*$-algebras with dimension ratio at most $R$,
     \item $S_n$, respectively $T_n$, contains $\F_1$, respectively $\F_2$, up to $\epsilon$,
     \item the smallest dimension of irreducible representation of $S_n$, respectively $T_n$, goes to infinity as $n \to \infty$.
 \end{enumerate}
 By \cite[Lemma 5.4]{EllNiu:MeanDimZero}, the tensor product $\mathcal{O}(\E_{\{x\}}) \otimes \mathcal{O}(\F_{\{y\}})$ can be locally approximated by subhomogenous $\mathrm C^*$-algebras with arbitrarily small dimension growth. It follows, as in the proof of Theorem~\ref{thm:Z-stable}, that $ \mathcal{O}(\E_{\{x\}}) \otimes \mathcal{O}(\F_{\{y\}})$ is $\mathcal{Z}$-stable. Let $X_0 \subset X$ be a closed subset meeting every $\alpha$-orbit at most once such that for every $N \in \mathbb{Z}_{>0}$ there exists an open set $V_N \supset X_0$ such that $\mathscr{V}|_{\alpha^n(V_N)}$ is trivial for every $-N \leq n \leq N$. Let $Y_0 \subset Y$ be a closed subset meeting every $\beta$-orbit at most once such that for every $N \in \mathbb{Z}_{>0}$ there exists an open set $W_N$ such that $\mathscr{W}|_{\alpha^n(W_N)}$ is trivial for every $-N \leq n \leq N$. By \cite[Proposition 2.5]{ArchPhil:SR1} with \cite[Theorem 6.14]{AAFGJSV2024}, $\mathcal{O}(\E_{X_0}) \otimes \mathcal{O}(\F_{Y_0})$ and $\mathcal{O}(\E_{x}) \otimes \mathcal{O}(\F_{y})$ are centrally large subalgebras of $\mathcal{O}(\E) \otimes \mathcal{O}(\F)$. It follows that $\mathcal{Z}$-stability passes from $ \mathcal{O}(\E_{\{x\}}) \otimes \mathcal{O}(\F_{\{y\}})$ to $ \mathcal{O}(\E) \otimes \mathcal{O}(\F)$, and from $ \mathcal{O}(\E) \otimes \mathcal{O}(\F)$ to $\mathcal{O}(\E_{X_0}) \otimes \mathcal{O}(\F_{Y_0})$. Thus if $A, B \in \mathcal{C}$, we have $A \otimes B \cong A \otimes B \otimes \mathcal{Z}$.
\end{proof}

This gives us the following theorem.

\begin{theorem} \label{thm:classificationTP}
 Suppose that $A, B, C, D\in \mathcal{C}$ and 
	\[ \psi : \Ell(A \otimes B) \to \Ell(C \otimes D) \]
	is an isomorphism. Then there exists a $^*$-isomorphism 
	\[ \Psi : A \otimes  B \to C \otimes D,\]
satisfying $\Ell(\Psi) = \psi$.
\end{theorem}

We finish with an example. Let us briefly recall the definition of the external tensor product of Hilbert bimodules.

    Suppose $A$ and $B$ are $C^*$-algebras, that $X$ is a Hilbert $A$-bimodule and $Y$ is a Hilbert $B$-bimodule. Then the algebraic tensor product $X \otimes_\mathbb{C} Y$ is an $A \otimes_{\mathbb{C}} B$-bimodule where the right and left actions are defined on simple tensors by
    \[ (a\otimes b)(x \otimes y) = (ax \otimes b y), \quad (x \otimes y) (a\otimes b) = (xa \otimes yb),\]
    for $a \in A$, $b \in B$, $x \in X$, and $y \in Y$. 
    There are unique $A \otimes_\mathbb{C} B$-valued pre-inner products on $X \otimes_\mathbb{C} Y$ such that
\begin{align*}
&_{A \otimes B} \langle x_1 \otimes y_1, x_2 \otimes y_2 \rangle=_A\langle x_1,x_2 \rangle \otimes _B\langle y_1, y_2 \rangle,\\
& \langle x_1 \otimes y_1, x_2 \otimes y_2 \rangle_{A \otimes B}=\langle x_1,x_2 \rangle_A \otimes \langle y_1, y_2 \rangle_B,
\end{align*}
for every $x_1, x_2 \in X$ and $y_1,y_2 \in Y$.
Then the completion of $X \otimes_\mathbb{C} Y$, denoted by $X \boxtimes Y$, is a Hilbert $A \otimes B$-bimodule.

\begin{example} In \cite{fowler}, Fowler defines a \emph{product system} of Hilbert bimodules $\{X_p\}_{p \in P}$ over a semigroup $P$, and from this data constructs a generalized Cuntz--Pimsner algebra. In the case that the semigroup is a group $G$, one can think of a product system as group action by Hilbert bimodules, and the resulting $\mathrm C^*$-algebra as a type of crossed product. When the (semi)group is $\mathbb N$ or $\mathbb Z$, one recovers the usual Cuntz--Pimsner algebra $\mathcal O_{X_0}(X_1)$.

When the group is $\mathbb Z^2$, a product system consists of a family
\[
X = \{\, X_{\mathbf n} : \mathbf n \in \mathbb Z^2 \,\}
\]
of Hilbert $A$-bimodules, together with a collection of $A$-bimodule isomorphisms
\[
\mu_{\mathbf m, \mathbf n} : X_{\mathbf m} \otimes_A X_{\mathbf n} 
\;\xrightarrow{\;\cong\;}\; X_{\mathbf m +\mathbf n},
\]
such that the following conditions hold:

\begin{enumerate}[label=\textup{(\arabic*)}]
\item $X_{\mathbf 0} = A$, and for each $\mathbf n \in \mathbb Z^2$,
\[
\mu_{\mathbf 0,\mathbf n}(a \otimes \xi) = a \cdot \xi, 
\quad
\mu_{0, \mathbf n}(\xi \otimes a) = \xi \cdot a.
\]

\item For every $\mathbf m, \mathbf n,\mathbf p \in \mathbb Z^2$, we have
\[
\mu_{\mathbf m + \mathbf n,\mathbf p} \circ (\mu_{\mathbf m, \mathbf n} \otimes \mathrm{id})
\;=\;
\mu_{\mathbf m, \mathbf n + p} \circ (\mathrm{id} \otimes \mu_{\mathbf n, \mathbf p}),
\] 

\item   For every $\mathbf n\in \mathbb Z^2$, $X_{- \mathbf n} \;\cong\; X_{\mathbf n}^*.$

\end{enumerate}

     Let $(X, \alpha)$ and $(Y, \beta)$ be a minimal dynamical systems of arbirtary mean dimension.  Choose line bundles $\mathscr V$ over $X$, $\mathscr W$ over $Y$ and let $\mathcal E :
    = \Gamma(\mathscr{V}, \alpha)$ and $\F := \Gamma(\mathscr{W}, \beta)$. 

 Let $Y = X \times X$.   For $(m,n) \in \mathbb Z^2$, let
    \[ B_{(1,0)} := \E \boxtimes C(X), \quad  B_{(0,1)} := C(X) \boxtimes \F, \]
where the $\boxtimes$ denotes the external tensor product of Hilbert bimodules. In particular, $B_{(1,0)}$ and $B_{(0,1)}$ are Hilbert $C(Y)$-bimodules.

For $(m,n) \in \mathbb Z^2$, define the Hilbert $C(Y)$-bimodule
\[
B_{(m,n)} = B_{(1,0)}^{\otimes m} \otimes_{C(Y)} B_{(0,1)}^{\otimes n},
\]
where the tensor product superscripts denoted the $m$- and $n$-fold interior tensor product of the $C(Y)$-modules of  $B_{(1,0)}$ and $B_{(0,1)}$, respectively.

If $X_1, X_2, Y_1, Y_2$ are any $C(X)$-bimodules, we have
\[  (X_1 \boxtimes X_2) \otimes_{C(Y)} (Y_1 \boxtimes Y_2) \cong (X_1 \otimes_{C(X)} Y_1) \boxtimes (X_2  \otimes_{C(X)} Y_2), \]
via the map satisfying
\[ (x_1 \boxtimes x_2) \otimes (y_1 \boxtimes y_2) \mapsto (x_1 \otimes y_1) \boxtimes (x_2 \otimes y_2). \]

 From this we have that 
 \[
 B_{(m,n)} \otimes_{C(Y)} B_{(m',n')} \cong  B_{(m',n')} \otimes_{C(Y)} B_{(m,n)} \cong B_{(m+m', n+n')}.
 \] 
 
 Then it is straightforward to check 
\[
\mathcal B = \{B_{(m,n)} \}_{(m,n) \in \mathbb Z^2}
\]
is a product system over $\mathbb Z^2$  with unit fibre $B_{(0,0)} \cong C(Y)$.

A \emph{covariant representation} of $\mathcal B$ in a $C^*$-algebra $C$ is a map $\psi : \mathcal B \to C$ such that for every $\psi_{(m,n)} := \psi|_{B_{(m,n)}}$, we have that $(\psi_{(0,0)}, \psi_{(m,n)})$ is a covariant representation of the Hilbert $B_{(0,0)}$-bimodule $B_{(m,n)}$, in the sense of ~\ref{sec:prelim}0.

Let $\mathcal O_B$ denote the associated generalized Cuntz--Pimsner algebra as in \cite[Proposition 2.9]{fowler}, which is generated by the universal covariant representation of $\mathcal B$, $\psi^{(u)}$.

Note that 
\[
(\psi_{(0,0)} \circ \id \otimes 1_{C(X)}, \psi_{(1,0)} \circ \id_{\E} \boxtimes 1_{C(X)})
\]
is a covariant representation of $\E$, and
\[
(\psi_{(0,0)} \circ 1_{C(X)} \otimes \id_{C(X)}, \psi_{(0,0)} \circ  1_{C(X)}\boxtimes \id_\F)
\]
is a covariant representation of $\E$. Thus there are surjective maps 
\[
\varphi_\E : \mathcal O(\E) \to C^*( \psi_{(0,0)} \circ \id \otimes 1_{C(X)}, \psi_{(1,0)} \circ \id_{\E} \boxtimes 1_{C(X)}) \subset \mathcal O_\mathcal{B},
\]
and 
\[
\varphi_\F: \mathcal O(\F) \to C^*( \psi_{(0,0)} \circ 1_{C(X)} \otimes \id_{C(X)}, \psi_{(0,0)} \circ  1_{C(X)}\boxtimes \id_\F) \subset \mathcal O_\mathcal{B},
\]
whose images commute. It follows that 
\[
\varphi_\E \otimes \varphi_\F : \mathcal O (\E) \otimes \mathcal O(\F) \to \mathcal O_\mathcal{B},
\]
is a well-defined $^*$-homomorphism, which is injective since $\mathcal O (\E) \otimes \mathcal O(\F)$ is simple. Moreover, it is easy to see that $\varphi(\E)$ and $\varphi(\F)$ generate $\mathcal O_{\mathcal {B}}$, hence $\varphi_\E \otimes \varphi_\F$ is also surjective. 

It follows that $\mathcal O_{\mathcal B} \cong \mathcal O(\E) \otimes \mathcal O(\F)$, and so  $\mathcal O_{\mathcal B}$ is $\mathcal Z$-stable.
\end{example}


\end{document}